\documentclass{article}

\usepackage{amssymb,amsmath,latexsym,amsthm}

\newfont{\Bb}{msbm10 scaled\magstep0}
\newfont{\Bbl}{msbm10 scaled\magstep1}
\newfont{\Bbs}{msbm10 scaled 800}

\newcommand{\fq}{\mbox{\Bb{F}}_q}
\newcommand{\f}{\mbox{\Bb{F}}}
\newcommand{\fp}{\mbox{\Bb{F}}_p}
\newcommand{\sfq}{\mbox{\Bbs{F}}_q}
\newcommand{\fqk}{\mbox{\Bb{F}}_{q^k}}

\newcommand{\Z}{\mbox{\Bb{Z}}}
\newcommand{\Q}{\mbox{\Bb{Q}}}

\newcommand{\Hyper}{\mbox{\Bb{H}}}
\newcommand{\ord}{{\rm ord}}
\newcommand{\Proj}{{\rm Proj}}

\newcommand{\W}{W} 
\newcommand{\K}{K} 

\newcommand{\Xf}{X}
\newcommand{\Df}{D}
\newcommand{\Uf}{U}
\newcommand{\Vf}{V}
\newcommand{\Sf}{S}
\newcommand{\Zf}{Z}
\newcommand{\Ef}{\tilde{D}}
\newcommand{\Yf}{\tilde{X}}
\newcommand{\Uftil}{\tilde{U}}

\newcommand{\Yhyp}{Y} 

\newcommand{\XW}{{\mathcal X}} 
\newcommand{\DW}{{\mathcal D}}
\newcommand{\UW}{{\mathcal U}}
\newcommand{\VW}{{\mathcal V}}
\newcommand{\SW}{{\mathcal S}}
\newcommand{\ZW}{{\mathcal Z}}

\newcommand{\YW}{\tilde{{\mathcal X}}}
\newcommand{\EW}{\tilde{{\mathcal D}}}
\newcommand{\UWtil}{\tilde{{\mathcal U}}}

\newcommand{\XK}{{\mathcal X}_K}
\newcommand{\DK}{{\mathcal D}_K}
\newcommand{\UK}{{\mathcal U}_K}
\newcommand{\VK}{{\mathcal V}_K}
\newcommand{\SK}{{\mathcal S}_K}
\newcommand{\ZK}{{\mathcal Z}_K}
\newcommand{\EK}{\tilde{{\mathcal D}}_K}

\newcommand{\YK}{\tilde{{\mathcal X}}_K}

\newcommand{\pr}{\mbox{\Bb{P}}}
\newcommand{\spr}{\mbox{\Bbs{P}}}
\newcommand{\Aff}{\mbox{\Bb{A}}}

\newcommand{\Hprim}{H(\Xf)_{prim}} 
\newcommand{\Hlog}{H(\Xf,\log \Df)} 
\newcommand{\Hlogk}{H(\Xf,k \Df)}
\newcommand{\HX}{H(\Xf)} 

\newcommand{\wt}{{\rm wt}} 
\newcommand{\Oh}{{\mathcal O}}  
\newcommand{\lcm}{{\rm lcm}} 
\newcommand{\CurlyH}{{\mathcal H}} 
\newcommand{\Spec}{{\rm Spec}}

\newcommand{\Tr}{{\rm Tr}}
\newcommand{\B}{{\mathcal B}}
\newcommand{\rk}{{\rm rank }\,}

\newtheorem{theorem}{Theorem}[section]
\newtheorem{lemma}[theorem]{Lemma}
\newtheorem{corollary}[theorem]{Corollary}
\newtheorem{proposition}[theorem]{Proposition}

\newtheorem{definition}[theorem]{Definition}

\newenvironment{exafont}{\begin{bf}}{\end{bf}}

\newenvironment{note}{\vspace{0.3cm}\par\noindent\refstepcounter{theorem}\begin{exafont}Note
    \thetheorem\end{exafont}\hspace{\labelsep}}{\vspace{0.3cm}\par}

\title{Ranks of elliptic curves over function fields}

\author{Alan G.B. Lauder}

\begin{document}

\maketitle

\begin{abstract}
We present experimental evidence to support the widely held belief that one half of all elliptic curves have infinitely
many rational points. The method used to gather this evidence is a refinement of an algorithm due to the author which
is based upon rigid and crystalline cohomology.
\end{abstract}

\section{Introduction}\label{Sec-Intro}

There is a widely held belief that one half of all elliptic curves have infinitely many rational points, but all experimental data which has been collected so far suggests that the fraction is actually two thirds
\cite[Introduction, Items (1) and (2)]{BMSW}. The final purpose of this paper is to present experimental evidence supporting the widely held belief. Let us first digress, to explain the origin of the method which has been used to collect this evidence.

In an earlier paper, the present author introduced a new method for computing zeta functions of varieties over finite fields \cite{L}. The principal novelty of the method was that it proceeded by induction on the dimension of the variety. The author named it the {\it fibration method}. It was observed that the method appeared to be especially useful for the case of surfaces over finite fields which can be fibred into low genus curves. Given such a surface, the interesting part of its zeta function is exactly the L-function associated to the generic fibre --- the generic fibre is a curve over a function field. According to a conjecture of Tate \cite{JT0}, the L-function reveals arithmetic properties of the curve, such as the rank of the group of rational points on its Jacobian. In the present paper, we carefully refine the fibration method, focussing ultimately on the case of elliptic curves over the rational function field, and use it to compute the L-functions of a randomly chosen sample of elliptic curves of extremely large conductor from a certain family. The computations reveal that around one half of those have rank, or at least analytic rank, greater than zero; see Section \ref{Sec-Ranks}. 

Leaving aside this final motivating application, the central achievement of this paper is the further development of the fibration method, with an emphasis
upon the computation of L-functions of hyperelliptic curves over the rational function field. We shall now sketch the improvements to the method which are obtained. We need to introduce some notation: For $q$ a power of a prime $p$, let $\fq$ denote the finite field with $q$ elements
of characteristic $p$, $\K$ denote the unique unramified extension of the field of
$p$-adic numbers  of degree $\log_p(q)$, and $\W$ denote the ring of integers of $\K$.

The challenge which we address
is that of efficiently computing the zeta function of a smooth variety over a finite field $\fq$. Our methods are
based upon rigid and crystalline cohomology. 
Using standard techniques, one sees that the problem of central interest is that of calculating the 
$p$th power Frobenius action on the
middle-dimensional rigid cohomology of the variety. This is a finite-dimensional vector space
over $K$. In the case in which the original
variety $\Vf$ of dimension $n$ can be fibred into smooth varieties $\Vf \rightarrow \Sf$ over a curve
$\Sf \subset \pr^1$, the paper \cite{L} presents a method of reducing the
problem of computing Frobenius on $H^n_{rig}(\Vf)$ to that of computing Frobenius on
$H^{n-1}_{rig}(\Vf_0)$ for one smooth fibre $\Vf_0 \hookrightarrow \Vf$ in the family. So the method
proceeds by induction on the dimension. We now describe how the work in this paper refines 
the fibration method.

First, the method requires that a basis be found for
the relative cohomology space $H^{n-1}_{rig}(\Vf/\Sf)$ such that the ``matrix for the Gauss-Manin
connection is Fuchsian'' \cite[Section 4]{L}. This is an entirely reasonable assumption as such bases are known to 
exist for families, for example, with an irreducible monodromy group, at least after extending the base field; however, the problem of actually finding such bases is an interesting open one, related to Hilbert's 21$^{st}$ problem. An exposition of this problem is given, and modest contributions made, in Sections \ref{Sec-Fuchs} and \ref{Sec-ConnMat}; however, this still remains an obstacle to the wider application of the fibration method.

Second, in \cite[Section 9.3]{L} some new ideas on how to significantly improve the practical
performance of the fibration method are sketched, although no details are offered. We work
out these ideas in great detail and use them in our calculations for elliptic curves. Specifically,
let $\Xf$ be some smooth compactification of the variety $\Vf$. There is a finitely generated
$\W$-module $H^n_{cris}(\Xf)$
attached to $\Xf$ called the middle-dimensional crystalline cohomology, and a natural map $H^n_{cris}(\Xf) \rightarrow H^n_{rig}(\Vf)$. The idea
is that one computes a basis for the Frobenius-invariant $\W$-lattice $\Hprim = {\rm Im}(H^n_{cris}(\Xf) \rightarrow H^n_{rig}(\Vf))$ and calculates Frobenius on this basis.
Working with this basis, one can exploit the Hodge filtration to significantly reduce the $p$-adic
precision to which the matrix for Frobenius needs to be computed. Rather than find $\Hprim$ itself, 
in Section \ref{Sec-CrysBases} and \ref{Sec-deJong} for certain surfaces $\Xf$
we describe how to compute a lattice $\Hlogk$ (where $\Df$ is a smooth divisor and $k$ a suitably
large integer) which is very close, in a precise sense, to $\Hprim$, although possibly of larger rank. This lattice is sufficient for our application and easier to find. 

Third, in \cite[Conjecture 7.4]{L} a conjecture is made on the vanishing of certain coefficients in the
``relative Frobenius matrix'' attached to families $\Vf \rightarrow \Sf$ of hyperelliptic curves and a choice of basis for $H^1_{rig}(\Vf/\Sf)$. The truth of this conjecture immediately yields improvements
to the running time of the original fibration method, see \cite[Examples 9.2, 9.3]{L}, and also the refined method of this paper. We prove
the conjecture. This also gives a practical improvement when one uses the deformation method to
compute zeta functions of hyperelliptic curves over finite fields. 

We now explain in greater detail the main theoretical contribution of this paper. Namely,
 our construction and application of
$\W$-lattices sitting in the middle-dimensional rigid cohomology of surfaces. This is the second
refinement to the fibration method as introduced above.

For $\Uf$ a smooth surface of degree $d$ in $\Aff^3_{\sfq}$, under certain
assumptions for each $k \geq \max\{2d-4,0\}$ 
we explicitly construct a full rank $\W$-lattice $\Hlogk$ sitting in the $\K$-vector space
$H^2_{rig}(\Uf)$. The motivation for constructing this lattice is Theorem \ref{Thm-Surprise}. This theorem
asserts that when $q = p$ is prime, if we compute an approximation $\tilde{A}$ to the matrix
$A$ for the $p$th power Frobenius map $F$ with respect to a basis for this lattice, the loss of precision in the calculation of $\det(1 - p^{-1} \tilde{A} T)$ is bounded by $h^{2,0} + \lfloor \log_p(k+1) \rfloor$. Here
$h^{2,0}$ is the geometric genus of a smooth compactification of a lifting to $\W$ of $\Uf$. Using an arbitrary basis of $H^2_{rig}(\Uf)$, even when $\tilde{A}$ has entries in $\W$ the author does not know how to improve the naive
loss of precision bound of $\dim(H^2_{rig}(\Uf)) - 1$ (or $\lfloor \dim( H^2_{rig}(\Uf))/2 \rfloor - 1$ using the functional equation).
For the surfaces we implicitly consider
in Section \ref{Sec-Ranks}, we have $p = 7$ and $(h^{2,0},\dim(H^2_{rig}(U)),d)  \in \{(0,8,6),(1,20,12),(2,32,18),(3,44,24),(4,56,30)\}$. So the use of a basis for $\Hlogk$ makes a very significant impact on precisions required.
Lying behind Theorem \ref{Thm-Surprise} is the existence of an embedding $\Hlog \rightarrow \Hlogk$ with
cokernel killed by multiplication by $p^{\lfloor \log_p(k+1) \rfloor}$, where $\Hlog \subset H^2_{rig}(\Uf)$
is a full rank $\W$-lattice endowed with what we call a Frobenius-Hodge structure (Definition \ref{Def-FHS}).
In the case in which $\Uf$ compactifies to a smooth surface $\Xf$ in $\pr^3_{\sfq}$ with smooth
divisor at infinity, then our lattice
$\Hlog$ is exactly the log-crystalline cohomology of the pair $(\Xf,\Df)$, where $\Df := \Xf \backslash \Uf$
(Section \ref{Sec-CrysBases}). Unfortunately, for the surfaces we are mainly interested in
(Sections \ref{Sec-Hyper} and \ref{Sec-Ranks}), this compactification is not smooth. The surfaces $\Uf$
we are mainly interested in do compactify to quasi-smooth surfaces $\Xf$ in weighted projective spaces. We make the additional assumption that these weighted projective surfaces are quotients of smooth surfaces
$\Yf$ in $\pr^3_{\sfq}$ with smooth divisor at infinity $\Ef$ under a natural finite group action. The lattices $\Hlog$ and $\Hlogk$ are then
{\it defined} by taking the invariant part under this group action of the lattices 
$H(\Yf,\log \Ef)$ and $H(\Yf, k\Ef)$
associated to the affine part $\Uftil := \Yf \backslash \Ef$ of the smooth projective surface $\Yf$ (Section \ref{Sec-LatticeWPS}). 
The lattices $\Hlog$ and $\Hlogk$ sit naturally in $H^2_{rig}(\Uf)$, and
we prove that $\Hlog$ has a Frobenius-Hodge structure (Section \ref{Sec-FHW}).
However, we stress that we do not prove that $\Hlog$ is the log-crystalline cohomology of
anything (although it undoubtedly is), see Note \ref{Note-IsLC}. 

The second and third refinements to the fibration method do not impact in any dramatic manner
on the asymptotic analysis of running time of the method; see Section \ref{Sec-Complexities} for a
discussion. Their significance is perhaps best appreciated with reference to explicit examples.
Regarding the second refinement, in \cite[Example 9.1]{L} a nine-fold speed-up is obtained
by assuming the truth of \cite[Conjecture 7.4]{L}. Regarding the third refinement,
in Section \ref{Sec-Ranks} the author computes, amongst other things, the L-functions
of 1000 elliptic curves over $\f_7(t)$ which have discriminant a squarefree polynomial of degree $60$.
For each curve, the computation took around 1 hour using the second and third refinements in this paper (the first is not relevant for the curves chosen). With only the second refinement, each computation
would have taken around 4 hours.
We note that in our computations in Section \ref{Sec-Ranks}, and the timing presented in this paragraph,
we use an algorithm due to K. Kedlaya for recovering Weil polynomials from their $p$-adic approximations.
\cite[Section 4]{KK1}. The output of our implementation of Kedlaya's algorithm is provably correct, although
there are no explicit bounds on its running time. All of our calculations are done using the computer algebra system {\sc Magma}.

We conclude the introduction with a brief mention of the contents of the other sections in this paper.
Section \ref{Sec-Refined} gives an overview of our refined fibration method and its relevance to the computation of L-functions. Section \ref{Sec-UtoV} proves some theorems of modest interest relating to one of the steps in the method, and
includes a short discussion of the computation of the residue map.

The author has been assisted in different ways by many different people, among them are:
Noam Elkies, Victor Flynn, Ralf Gerkmann, Keith Gillow, Roger Heath-Brown, Hendrik Hubrechts, Nicholas Katz, 
Bernard Le Stum, Atsushi Shiho, William Stein, Balazs Szendroi and Mark Watkins. He would like to thank them all. 
Especial thanks are due to Jan Denef, Kiran Kedlaya and Nobuo Tsuzuki, who answered with
great patience many questions from the author. Finally, a great debt of gratitude is owed to 
A.J. de Jong and Bjorn Poonen.

\section{The refined fibration method}\label{Sec-Refined}

The purpose of this section is to give an outline of the fibration method. The outline is based
upon the original method presented in \cite{L}, incorporating refinements to the improvements
crudely sketched in \cite[Section 9.3]{L}.
 To deal with the method in full would require the introduction of a great burden of notation, so we shall focus on aspects of it that
are developed in this paper, and refer to \cite{L} for details on parts which are fully investigated in
that paper.

\subsection{$p$-adic rings and finite fields}

Let us first fix some notation which will be used throughout the paper. Let $\fq$ be the finite field
with $q$ elements of characteristic $p$, $\W := W(\fq)$ be the ring of Witt vectors of
$\fq$, and $\K$ be the fraction field of $\W$. Thus $\K$ is the unramified extension of the field
of $p$-adic
numbers of degree $\log_p(q)$, and $\W$ is the ring of integers of $\K$.
Let $\sigma: \W \rightarrow \W$  and $\sigma_K : \K \rightarrow \K$ be the Frobenius automorphisms
on $\W$ and $\K$ respectively, that is, the maps induced by functoriality by the $p$th power map on
$\fq$.  Let
$\bar{\f}_q$ denote an algebraic closure of $\fq$, and for
each integer $k \geq 1$, let $\fqk$ denote the unique extension of degree $k$ of
$\fq$ in $\bar{\f}_q$.

We warn the reader that the notation for the different varieties required in this paper is not
consistent with that in \cite{L}. The reason for this is that the focus of the two papers is quite
different, and the notation from \cite{L} would become burdensome in the present paper. 

\subsection{The zeta function and reduction to the middle dimension}

Let $\Xf$ be a smooth projective variety of dimension $n$ over the finite field $\fq$ of characteristic $p$. The 
{\it zeta function} $Z(\Xf,T)$ is the formal power series
\[ Z(\Xf,T) := \exp\left( \sum_{k = 1}^\infty |\Xf(\fqk)| \frac{T^k}{k}\right) \]
which encodes the number of $\fqk$-rational points $|\Xf(\fqk)|$ on $\Xf$ over the different
finite extensions $\fqk$ of the base field $\fq$. By a famous theorem of Dwork, $Z(X,T)$ is a rational function.
There is a formula from rigid cohomology \cite[Pages 7-9]{LeStum}
\[ Z(\Xf,T) = \prod_{i = 0}^{2n} \det(1 - T q^{n} F_q^{-1} | H^i_{rig}(\Xf))^{(-1)^{i+1}};\]
here each $H^i_{rig}(\Xf)$ is a finite dimensional $K$-vector space, and
$F_q$ a linear map induced by the $q$th power map on the structure
sheaf $\Oh_\Xf$ of $\Xf$. There is a factorisation $F_q = F^{\log_p(q)}$ where $F$ is the $\sigma_\K$-linear map on $H^i_{rig}(\Xf)$ induced by the $p$th power map on $\Oh_\Xf$.

The problem we consider is to compute $Z(\Xf,T)$ given equations defining the variety
$\Xf$.  We shall now try and convince the reader that the central problem is computing
the factor with $i = n$.

First, we observe that if $X$ is a hypersurface in a space whose zeta function is known, then
by Poincar\'{e} duality and the Lefschetz hyperplane theorem, one can reduce immediately
to the case $i = n$. So let us suppose we are not in such a fortunate situation. We shall
proceed by induction on the dimension, and assume one can compute zeta functions of smooth
zero dimensional varieties --- we offer no suggestions on how this base case should be tackled. 
By Poincar\'{e} duality, one need only calculate the factors in the formula for
$0 \leq i \leq n$.
Let us assume that $\Xf \subseteq \pr^m$, and there exists a smooth hyperplane section 
$\Yhyp \hookrightarrow \Xf$ defined over the finite field $\fq$ --- this is certainly the case after extending
the base field. The Lefschetz hyperplane theorem then reveals that the factors for $0 \leq i  < n - 1$
in the zeta function of $Z(\Yhyp,T)$ determine the corresponding factors in $Z(\Xf,T)$. So by induction we reduce to the cases
$i = n-1,n$. We suggest two different approaches to dealing with the case $i = n - 1$. First,
Deligne's pgcd theorem \cite[Th\'{e}or\`{e}me 4.5.1]{PD2} tells us that if one could compute the zeta functions of all
smooth hyperplane sections in some Lefschetz pencil on $\Xf$, then the $(n-1)$st factor could be recovered by computing a greatest common divisor. Unfortunately, the author is not aware of any effective bounds on how many
hyperplane sections need to be considered until one recovers the correct polynomial; in practice, though, this would seem a useful approach. Second, a more sophisticated, but rigorous,
method for obtaining the $(n-1)$st factor would be to use the spectral sequence in the
fibration method, c.f. the proof of Proposition \ref{Prop-Inj}. This should not be too difficult, although the author does not offer any details.
In any case, we hope the reader is convinced that the case $i = n$ is the one of main interest.

The problem on which we now focus is that of calculating a $p$-adic approximation to a matrix
for $F$ on the space $H^n_{rig}(\Xf)$.

\subsection{How the fibration method calculates in the middle dimension}\label{Sec-Middle}

\subsubsection{The geometric set-up}

Let $\XW$ be a smooth $\W$-scheme of relative dimension $n$ and $\DW$ a smooth
divisor on $\XW$. Assume that $\UW := \XW \backslash \DW$ is affine, and embed it
in some affine space $\Aff^m_\W$, $m \geq n+1$.
Choose a fibration $\Aff_\W^{m} \rightarrow \Aff^1_\W$. Let 
$\UW \rightarrow \Aff^1_\W$ be the restriction to $\UW$. Choose $\VW \subseteq \UW$ so that
the restriction $\VW \rightarrow \SW \subseteq \Aff^1_\W$ has smooth fibres, and define
$\ZW := \UW \backslash \VW$. By varying the embedding we may assume $\UW \rightarrow \Aff^1_\W$ is generically smooth, 
and so $\ZW$ is of codimension one.
Let $\VW_s$ for
$s \in \SW(\W)$ be some smooth fibre in the family $\VW \rightarrow \SW$ --- one may need to extend
the base field $\fq$ to find a smooth fibre.
Let the special fibres and generic fibres of $\XW,\DW,\UW,\VW,\ZW$ be denoted
$\Xf,\Df,\Uf,\Vf,\Zf$ and \linebreak $\XK,\DK,\UK,\VK,\ZK$, respectively. Note that $n$ is the dimension
of the special fibre $\Xf$. For example, if $\XW \subset \pr^m_\W$ then one can construct a Lefschetz pencil on $\XK$ with
axis lying in the divisor $\DK$. The singular fibres in $\UK \rightarrow \SK$ will each then
only have unique double-points.

\subsubsection{F-crystals and isocrystals}

\begin{definition}
An $F$-crystal is a finitely generated $\W$-module $M$ with a linear map
$F : \sigma^* M \rightarrow M$ which becomes an isomorphism after tensoring with $\K$. An
$F$-isocrystal is a finite dimensional $\K$-vector space $V$ with a linear map $F: \sigma_\K^* V \rightarrow V$ which
is an isomorphism.
 \end{definition}
 
That is, an $F$-crystal is a finitely generated $\W$-module $M$ with a $\sigma$-linear map
$F : M \rightarrow M$ which becomes an injection after tensoring by $\K$, and an $F$-isocrystal is a finite dimensional $K$-vector space with an injective $\sigma_K$-linear
map $F: V \rightarrow V$. Define the
$F$-crystal $\W(m)\,(m \in \Z,\, m \leq 0)$ to be the rank one module over $\W$ where $F$ acts on the basis
elements by multiplication by $p^{-m}$, and likewise for $K(m)\,(m \in \Z)$. We define Tate
twists $M(m) := M \otimes_{\W} \W(m)$ of an $F$-crystal $M$, and likewise for $F$-isocrystals.

Let $i$ be an integer. Associated to $\XW$ and $\DW$ are $F$-crystals
\[
 H^i_{cris}(\Xf), H^i_{cris}(\Df), \mbox{ and } H^i_{cris}((\Xf,\Df))
 \]
called the crystalline cohomology of $\Xf$ and $\Df$ and log-crystalline
cohomology of the pair $(\Xf,\Df)$, respectively.
The groups above fit together in an
$F$-equivariant long exact sequence \cite[Definition 2.3.3, Proposition 2.4.1, Proposition 2.2.8]{AKR}
\begin{equation}\label{Eqn-CrisExc}
 \cdots \rightarrow  H^i_{cris}(\Xf) \rightarrow H^i_{cris}((\Xf,\Df)) \rightarrow H^{i-1}_{cris}(\Df)(-1)
 \rightarrow \cdots 
\end{equation}
Associated to $\XW$, $\DW$, $\UW$, $\VW$, $\Zf$ and $\Vf_s$ are $F$-isocrystals
\[ H^i_{rig}(\Xf),\, H^i_{rig}(\Df),\,H^i_{rig}(\Uf),\, H^i_{rig}(\Vf), \mbox{ and }
H^i_{rig}(\Vf_s) \]
called the rigid cohomology of the special fibres, and $F$-isocrystals
$H^i_{rig,\Zf}(\Uf)$ called the rigid cohomology of $\Uf$ with support in $\Zf$. 
There is an $F$-equivariant
long exact sequence
\begin{equation}\label{Eqn-RigExc}
\cdots \rightarrow H^i_{rig,\Zf}(\Uf) \rightarrow H^i_{rig}(\Uf) \rightarrow H^i_{rig}(\Vf) \rightarrow \cdots .
\end{equation}
(The $F$-isocrystals $H^i_{rig,\Zf}(\Uf)$ are really just defined to make this sequence exact \cite[Section 4.4]{KKDuke}.)
There are natural maps
\[ 
\begin{array}{rcl}
H^i_{cris}(\Xf) & \rightarrow & H^i_{rig}(\Xf)\\
H^i_{cris}((\Xf,\Df)) & \rightarrow & H^i_{rig}(\Uf)
\end{array}
\]
which are injections modulo torsion with images full rank $\W$-lattices \cite[Definition 2.3.3, Proposition 2.4.1]{AKR}. We shall use the same symbol $F$ to denote the
Frobenius action on each of the $F$-crystals and $F$-isocrystals listed above.

\begin{definition}
The following images are Frobenius-invariant $\W$-lattices sitting naturally in $\K$-vector spaces; that is,
$F$-crystals embedded naturally in $F$-isocrystals.
\[
\begin{array}{lll}
\Hprim & := & {\rm Im}(H^n_{cris}(\Xf) \rightarrow H^n_{rig}(\Uf))\\
\Hlog & := & {\rm Im}(H^n_{cris}((\Xf,\Df)) \rightarrow H^n_{rig}(\Uf))\\
\HX & := & {\rm Im}(H^n_{cris} (\Xf) \rightarrow H^n_{rig}(\Xf)).
\end{array}
\]
\end{definition}
For the surfaces we consider, the lattice $\Hprim$ is exactly the free part of the primitive
middle-dimensional crystalline cohomology.

\subsubsection{A summary of the method}\label{Sec-FMSummary}

We wish to calculate a $p$-adic approximation to a matrix for the action of $F$ on 
the full rank Frobenius-invariant $\W$-lattice $\HX$.
The reason why it is important
to work with a basis for this lattice rather than an arbitrary basis is given in \cite[Section 9.3.2]{L}. Essentially, one can exploit the Hodge filtration to significantly reduce the power of $p$ to which the
matrix needs to be calculated. We first note that it is both easier and more efficient to
compute a matrix for $F$ on the $\W$-lattice $\Hprim$.
It is easier because $\Uf$ is affine and so the elements in $H^n_{rig}(\Uf)$ can be
described via $n$-forms defined on $\UK$. 
It is more efficient, at least for the case $n = 2$ of surfaces, 
since $\Hprim$ is the quotient of $\HX$ by the rank one
lattice generated by the class of the curve $\Df$, and the Frobenius
action on the class of this curve is just multiplication by $q$.

The steps in our algorithm for calculating $F$ on the lattice $\Hprim$ and our work upon the different steps can be tersely summarised as follows.
\begin{itemize}
\item{{\sc Step 1}: Find a set of $n$-forms in $\Gamma(\UK,\Omega^n_{\UK})$ whose image in
$H^n_{rig}(\Uf)$ is a basis for the full rank $\W$-lattice $\Hlog$. 
[This is the problem of finding explicit bases in
log-crystalline cohomology.  It is solved for surfaces in certain weighted projective
spaces in Sections \ref{Sec-CrysBases} and \ref{Sec-deJong}.]}
\item{{\sc Step 2}: Use the long exact sequence for crystalline and log-crystalline cohomology
(\ref{Eqn-CrisExc}) to compute a basis for the sublattice
$\Hprim \subseteq \Hlog$ from one for $\Hlog$ itself. [This is the problem of explicitly
computing the residue map. It is briefly discussed in Section \ref{Sec-Residue}.]
}
\item{{\sc Step 3}: Use the long exact sequence in rigid cohomology (\ref{Eqn-RigExc}) to map the basis
elements for $\Hprim$ into the space $H^n_{rig}(\Vf)$. Thus one reduces the problem to that of calculating
the action of $F$ on the space $H^n_{rig}(\Vf)$.
[We discuss this in Section \ref{Sec-UtoV} for certain surfaces.]
}
\item{{\sc Step 4}: Use the smooth fibration $\Vf \rightarrow \Sf$ to reduce the problem to that of 
calculating the action of $F$ on the space $H^{n-1}_{rig}(\Vf_s)$ for one smooth fibre $\Vf_s \hookrightarrow
\Vf$ defined over $\fq$.
[This is the heart of the fibration method, and is dealt with in detail in \cite[Sections 3, 4, 5]{L}. One outstanding problem which is not fully resolved in \cite{L} is that of finding a Fuchsian basis. This problem is explained in 
Section \ref{Sec-Fuchs}, and in Section \ref{Sec-ConnMat} a solution for elliptic curves is given using an algorithm due to W. Dekkers.]}
\end{itemize}
In practice, it is not always necessary or desirable to carry out Steps 1 and 2 exactly as stated.
First, it is easier though perhaps slightly less efficient to compute a matrix for $F$ on the full $\W$-lattice
$\Hlog$ rather than its sublattice $\Hprim$; that is, to omit Step 2. Second, it is not necessary to
find $\Hlog$ exactly, as a lattice which is very close to it will be sufficient. For example,
in our application to surfaces $\Xf$ we actually work with a lattice $\Hlogk$ ($k \gg 0$, see Theorem
\ref{Thm-P3SurfaceLattice} and Definition \ref{Def-Hlogk}) such that
\[ \Hprim \stackrel{i}{\hookrightarrow} \Hlog \stackrel{\rho}{\hookrightarrow}
\Hlogk \subseteq H^2_{rig}(\Uf)\]
where the cokernel of the embedding $\rho$ is killed by $p^{\lfloor \log_p(k+1) \rfloor}$. For certain
elliptic surfaces the map $i$ is an isomorphism and when $p$ is large enough relative to the
degree of the surface so is $\rho$, so we find
$\Hprim$ exactly. For the examples we present in Section \ref{Sec-Ranks}, the corank of the image of
$i$ is $2$, and the map $\rho$ has cokernel killed by at worst $p^2$. So $\Hlogk$ contains
a sublattice of corank $2$ which is very close to $\Hprim$. This is good enough for our application.
Theorem \ref{Thm-Surprise} exactly quantifies the benefits one obtains from working with
a basis of the lattice $\Hlogk$.

\subsection{L-functions}

We now explain the relevance of our algorithm to the computation of L-functions of certain curves.
Let $\Uf \subset \Aff_{\sfq}^3$ be a smooth affine surface and $\Uf \rightarrow \Aff_{\sfq}^1$ be the
fibration via one of the coordinate axes. Assume that one can extend $\Uf \rightarrow \Aff_{\sfq}^1$ to
a map $\hat{\Xf} \rightarrow \pr_{\sfq}^1$ where $\hat{\Xf}$ is a smooth projective surface. Assume that the
singular fibres of the map are geometrically irreducible, so the trivial part of the N\'{e}ron-Severi lattice of
$\hat{\Xf}$ is generated by the class of the zero section and a smooth fibre.
Define the $F$-invariant
$\W$-lattice
\[ H := {\rm Image}(H^2_{cris}(\hat{\Xf}) \rightarrow H^2_{rig}(\Uf)).\]
We shall define the L-function of the generic curve $\hat{\Xf}_\eta$ in the family $\hat{\Xf} \rightarrow \pr_{\sfq}^1$ 
to be
\[ L(\hat{\Xf}_\eta,T) := \det(1 - T F_q | H \otimes_{\W} \K).\]
This agrees with the usual definition, since it is exactly the ``interesting part'' of the
zeta function, obtained from $P_2(\hat{\Xf},T) := \det(1 - T F_q  | H^2_{rig}(\hat{\Xf}))$ 
by removing factors coming from the classes of the zero section and a smooth fibre in $H^2_{cris}(\hat{\Xf})$. 

Note that in our application we shall compactify $\Uf$ to a (quasi-smooth) surface $\Xf$ in weighted
projective space, and the lattice $\Hprim$ will then be defined (see Section \ref{Sec-FHW}) and equal to $H$ above.

\section{Crystalline cohomology of smooth surfaces in \Bbl{P}$^3$}\label{Sec-CrysBases}

In this section we address the problem of finding the lattice $
\Hlog = {\rm Im}(H_{cris}^n((\Xf,\Df)) 
\rightarrow H_{rig}^n(\Uf))$, in the case in which $\Xf$ is a smooth surface of degree $d$ sitting in projective
space $\pr^3_{\sfq}$. 
We shall proceed
in some generality in Sections \ref{Sec-CrysdeRham}, \ref{Sec-HyperLog} and \ref{Sec-HyperCoh}, before
specialising to smooth surfaces in $\pr^3_{\sfq}$ in Section \ref{Sec-Vanishing}. Ultimately, we will be satisfied
with finding a lattice $\Hlogk$, for any $k \geq \max\{2d-4,0\}$, which contains $\Hlog$ as a sublattice, with the quotient an abelian
group killed by multiplication by $p^{\lfloor \log_p(k+1) \rfloor}$. We use these lattices to define
analogous lattices in Section \ref{Sec-deJong} associated to certain weighted projective surfaces, and
explain why all of these lattices are useful in Section 
\ref{Sec-Hodge}.

\subsection{Crystalline and de Rham cohomology}\label{Sec-CrysdeRham}

Let $\XW$ be a smooth $\W$-scheme of relative dimension $2$ and
$\DW$ a smooth divisor on $\XW$. 
For $i \geq 0$, denote by $\Omega^i_{\XW} := \wedge^i \Omega^1_{\XW}$ the sheaf of differential $i$-forms on $\XW$, and by
\[
\Omega^\bullet_{\XW} : 0 \rightarrow \Oh_{\XW}
 \rightarrow \Omega_{\XW} \rightarrow  
\Omega^2_{\XW} \rightarrow 0 \]
the algebraic de Rham complex of $\XW$. The algebraic de Rham cohomology 
$H_{dR}(\XW)$ of $\XW$ is by definition the hypercohomology $\Hyper(\Omega^\bullet_\XW)$ of 
this complex. This is a finitely generated $\W$-module. The complex $\Omega_\DW^\bullet$ and
finitely generated $\W$-module $H_{dR}(\DW) := \Hyper(\Omega_\DW^\bullet)$ are
defined in an analogous manner.
For
$i \geq 0$, denote by $\Omega^i (\log \DW) := \wedge^i \Omega^1_\XW(\log \DW)$ the sheaf of differential $i$-forms on $X$ with logarithmic
poles along $\DW$, and by
\[ \Omega^\bullet_{\XW} (\log \DW) : 0 \rightarrow \Oh_\XW \rightarrow \Omega_\XW (\log \DW) 
\rightarrow \Omega^2_\XW (\log \DW) \rightarrow 0 \]
the logarithmic de Rham complex \cite[Definition 2.2.2]{AKR}.
The log-de Rham cohomology $H_{dR}((\XW,\DW))$ is by definition the
hypercohomology $\Hyper(\Omega^\bullet_\XW (\log \DW))$ of this complex.
This is a finitely generated $\W$-module.
There is an exact sequence 
\[ 0 \rightarrow \Omega^\bullet_\XW \rightarrow \Omega^\bullet_\XW (\log \DW) \stackrel{Res}{\rightarrow} j_* \Omega^\bullet_\DW[+1] \rightarrow 0 \]
where the map $Res$ is the residue map and $j: \DW \hookrightarrow \XW$ \cite[Proposition 2.2.8]{AKR}. 
This induces a long exact sequence on the hypercohomology
groups:
\[ \cdots  \rightarrow H_{dR}^i(\XW) \rightarrow H_{dR}^i((\XW,\DW)) \rightarrow H_{dR}^{i-1}(\DW) \rightarrow \cdots.\]
Let $\Xf$ and $\Df$ be the special fibres of $\XW$ and $\DW$ respectively.
Associated to the smooth pair of $\fq$-varieties $(\Xf,\Df)$ are finitely generated
$\W$-modules with canonical isomorphisms \cite[Definition 2.3.3, Proposition 2.4.1]{AKR}
\[ H_{cris}(\Xf) \cong H_{dR}(\XW),\,H_{cris}(\Df) \cong H_{dR}(\DW),\,H_{cris}((\Xf,\Df)) \cong
H_{dR}((\XW,\DW)).\]
These are the crystalline and log-crystalline cohomology groups. The aim of this section is to give an explicit presentation of the log-crystalline cohomology groups
$H_{cris}^i((\Xf,\Df))$ in terms of the homology of a complex of free $\W$-modules of finite rank.

\subsection{The hypercohomology of the logarithmic de Rham complex}\label{Sec-HyperLog}

For $i,k \geq 0$, let $\Omega^i (k \DW) := \Omega^i_\XW \otimes_{\Oh_\XW} \Oh_\XW(k\DW)$ denote the sheaf of
differential $i$-forms on $\XW$ with poles of order bounded by $k$ along $\DW$. Consider the
complex
\[ \Omega^\bullet_\XW((k + \bullet)\DW): 0 \rightarrow  \Oh_\XW (k\DW) \rightarrow
\Omega_\XW((k+1)\DW) \rightarrow \Omega^2_\XW((k+2)\DW) \rightarrow 0.\]

\begin{theorem}\label{Thm-KillCokernel}
The cokernels of the maps of homology sheaves induced by the natural map of complexes of
sheaves
\[ \Omega_\XW^\bullet (\log \DW) \hookrightarrow \Omega_\XW^\bullet((k + \bullet) \DW) \]
are killed by multiplication by $\lcm\{1,2,\dots,k+1\}$. Moreover, the induced maps on homology sheaves
are injective.
\end{theorem}

\begin{proof}
We will show using \'{e}tale local coordinates that for each $0 \leq i \leq 2$ the map of homology sheaves
\[ \CurlyH^i(\Omega_\XW^\bullet (\log \DW)) \rightarrow \CurlyH^i(\Omega_\XW^\bullet((k + \bullet) \DW)) \] 
is an injection and has cokernel killed by  $\lcm\{1,2,\dots,k+ i -1\}$. We shall use subscripts $x$ and $y$ to 
denote partial differentiation w.r.t. $x$ and $y$, respectively.

For points not lying on $\DW$ there is nothing to prove, since locally at these points
the sheaves are isomorphic. Etale locally around a point on $\DW$ the pair
$(\XW,\DW)$ looks like the hyperplane section $x = 0$ of an open subset $\Spec(R)$ of $\Spec(\W[x,y])$. 
By a further Zariski localisation 
we may assume $R$ contains a subring isomorphic to $R/(x)$, and derivation w.r.t. $x$ is trivial
on this subring. (That is, with $S := \{y_0\,|\, (0,y_0) \not \in \Spec(R)\}$, replace $R$ by
$R^\prime := R[1/(y - y_0); y_0 \in S]$ and then $\W[y][1/(y- y_0); y_0 \in S]$ is the required subring 
of $R^\prime$ isomorphic to $R^\prime/(x)$; then replace the notation $R^\prime$ by $R$.)
Let
$\ell$ be a positive integer, and define $n_\ell := \lcm\{1,2,\dots,\ell - 1\}$. For
$a = \sum_{i = -\ell}^{-1} a_i(y) x^i \in R[x^{-1}]$ where $a_i(y) \in R/(x) \subset R$, define
$\hat{a} := \sum_{i = -\ell}^{-2} (a_i(y)/(i+1)) x^{i+1} \in R[x^{-1}] \otimes_W \K$. Then
$\hat{a}_x = a + (a_{-1}(y)/x)$ and $n_\ell \hat{a} \in R[x^{-1}]$.

For $i = 0$ there is nothing to prove since the closed $0$-forms are just the constant functions for both sheaves. For $i = 2$, let $\alpha(x,y) dx \wedge dy \in \Omega^2_{R[x^{-1}]}$ be a closed $2$-form with a pole of order $k + 2$. Let $\alpha = a(x,y) + b(x,y)$ where $b$ has no pole along $\DW$, and write
$a$ as in the preceding paragraph, taking $\ell = k + 2$. Define $\beta := \hat{a} dy \in \Omega^1_{R[x^{-1}]} \otimes_\W \K$. Then
$\alpha dx \wedge dy - d \beta = \frac{a_{-1} dx}{x} \wedge dy + \gamma$ where the $2$-form $\gamma$ has no pole
along $\DW$. Morever, $n_{k+2} \beta \in \Omega_{R[x^{-1}]}$. Thus after multiplication by
$n_{k+2}$ locally any $2$-form can be written modulo exact $2$-forms as a $2$-form with a log-pole along 
$\DW$. This proves the claim on the cokernel for $i = 2$. 

We now consider the cokernel for $i = 1$. Let $\alpha = \beta(x,y) dx + \gamma(x,y) dy \in \Omega^1_{R[x^{-1}]}$ be a closed
$1$-form with a pole of order $k+1$ along $\DW$. Since $\alpha$ is closed,
$\beta_y = - \gamma_x$. As before, write $\beta = a(x,y) + b(x,y)$ where $b$ has no pole along $\DW$, and write $a$ as above, taking $\ell = k + 1$. Then $\alpha - d(\hat{a}) = \frac{a_{-1} dx}{x} + 
(\gamma - \hat{a}_y) dy$.  Using the
equation $\beta_y = - \gamma_x$ we see that $\gamma - \hat{a}_y$ has no pole
along $\DW$. Thus after multiplication by
$n_{k+1}$ locally any closed $1$-form can be written modulo exact $1$-forms as a $1$-form with a log-pole along 
$\DW$. 

We now prove the maps on sheaves are injective by examining kernels. For $i = 0$ there is nothing 
to prove. For $i = 1$, a closed $1$-form $\alpha \in \Omega^1_{R[x^{-1}]}$ with a log-pole has the form $\beta (x,y) \frac{d x}{x} + 
\gamma(x,y) dy$ where the functions $\beta, \gamma \in R$ have no poles along $\DW$, with
$(\beta/x)_y = - \gamma_x$. Suppose that $\alpha = df$ for some function $f(x,y) \in R[x^{-1}]$ with a pole of
order $\ell \geq 1$ along $\DW$. Since $f_y = \gamma$, we see $f_y$ has no pole along $\DW$, so $f$ must be
independent of $y$ with $\gamma = f_y = 0$. Thus $(\beta/x)_y = 0$, so $\beta$ does not depend on $y$, and we have
$f(x)_x = \beta(x)/x$. This can only be true if $\beta(x)/x$ has no pole along $\DW$, and so
neither does $f$, a contradiction. Thus no such $f$ can exist, and the map on homology sheaves is
injective for $i = 1$. 

Now let $i = 2$. A (closed) $2$-form $\alpha \in \Omega^2_{R[x^{-1}]}$ with a log-pole locally has the form
$\beta (x,y) \frac{dx}{x} \wedge dy$ for some function $\beta \in R$ with no pole along $\DW$. Suppose that 
$\alpha = d \gamma$ where the $1$-form $\gamma \in \Omega^1_{R[x^{-1}]}$ has a pole of order $\ell \geq 1$ along $\DW$. We shall show that the exists a $1$-form $\gamma^{\prime} \in \Omega^1_{R[x^{-1}]}$ with a log-pole along $\DW$ such that
$\alpha = d \gamma^{\prime}$. This proves our map of homology sheaves is injective on
$2$-forms.

Write $\gamma = a(x,y) dx + b(x,y) dy$. Let $a = a^{\prime \prime} + a^{\prime}$ and $b = b^{\prime \prime} + b^{\prime}$ where
$a^{\prime}$ has a simple pole along $\DW$ with $a^{\prime \prime} := a - a^{\prime} = \sum_{i = -\ell}^{-2} a_i(y) x^i$,
and $b^{\prime}$ has no pole along $\DW$ with $b^{\prime \prime} := b - b^{\prime} = \sum_{i = -\ell}^{-1} b_i(y) x^i$ . Then since $\beta(x,y)/x = b_x - a_y$ has a log-pole, we see  
$b_{-\ell} = 0$, 
and for $i = -\ell,-\ell + 1,\dots,-2$ we have $(a_{i})_y = (i + 1)b_{i+1}$.
Hence $\gamma^{\prime \prime} := a^{\prime \prime} dx + b^{\prime \prime} dy$ is a closed $1$-form. Defining
$\gamma^{\prime} := a^{\prime} dx + b^{\prime} dy$ one sees that
$\gamma^{\prime}$ has a log-pole along $\DW$. Moreover,
$\alpha = d \gamma = d(\gamma^\prime  + \gamma^{\prime \prime}) = d\gamma^{\prime}$, as required.
\end{proof}

\begin{note}
An analogue of the claim on the cokernel is proved in \cite[Theorem 2.2.5]{AKR} with the complex
$\Omega_\XW^\bullet ((k + \bullet) \DW)$ replaced by the twisted logarithmic
complex whose $i$th term is $\Omega^i _\XW (\log \DW) \otimes_{\Oh_\XW} \Oh_\XW (k \DW)$, 
for the more general situation in which $\DW$ is a smooth normal crossings divisor in a smooth
$\W$-scheme $\XW$ of arbitrary dimension $n$.
\end{note}

Define the complex of sheaves $Q^\bullet$ on $\XW$ so that the sequence 
\begin{equation}\label{Eqn-Q}
 0 \rightarrow \Omega_\XW^\bullet (\log \DW) \rightarrow \Omega_\XW^\bullet
((k + \bullet) \DW) \rightarrow Q^\bullet \rightarrow 0
\end{equation}
is exact. 

\begin{corollary}
The homology sheaves $\CurlyH^i(Q^\bullet)$ of the complex $Q^\bullet$ are killed by multiplication by
$\lcm\{1,2,\dots,k+1\}$.
\end{corollary}

\begin{proof}
The long exact sequence for homology sheaves from the short exact sequence (\ref{Eqn-Q}) has the form
\[ \cdots \rightarrow \CurlyH^i(\Omega_\XW^\bullet (\log \DW)) \stackrel{\theta_i}{\rightarrow} \CurlyH^i( \Omega_\XW^\bullet
((k + \bullet) \DW)) \stackrel{\phi_i}{\rightarrow} \CurlyH^i(Q^\bullet) \]
\[ \stackrel{\psi_i}{\rightarrow} \CurlyH^{i+1}(\Omega_\XW^\bullet (\log \DW))
\stackrel{\theta_{i+1}}{\rightarrow} \CurlyH^i(\Omega_\XW^\bullet
((k + \bullet) \DW)) \stackrel{\phi_{i+1}}{\rightarrow} \cdots \]
Since $\theta_{i+1}$ is an injection by Theorem \ref{Thm-KillCokernel}, $\CurlyH^i(Q^\bullet)$ is isomorphic to the cokernel of $\theta_i$ and thus by that theorem is killed as claimed.
\end{proof}

In the nicest situation one has the following corollary to Theorem \ref{Thm-KillCokernel}.

\begin{corollary}\label{Cor-KillQ}
If the characteristic $p$ of the residue field of $\W$ is strictly greater than
$k + 1$, then we have the following isomorphism of hypercohomology groups:
\[ H_{dR}((\XW,\DW)) := \Hyper (\Omega^\bullet_\XW (\log \DW)) \cong
\Hyper (\Omega^\bullet_\XW ((k + \bullet) \DW)).\]
\end{corollary}

\begin{proof}
Since in this case $\lcm\{1,2,\dots,k+1\}$ is invertible in $\W$, for each $i$ by Theorem \ref{Thm-KillCokernel} we have an isomorphism of homology sheaves
$\CurlyH^i(\Omega^\bullet_\XW (\log \DW)) \cong \CurlyH^i(\Omega^\bullet_\XW ((k + \bullet)
\DW))$. These homology sheaves are the first terms in a spectral sequence 
$E_2^{p,q} := H^p(\XW,\CurlyH^p(\star))$
computing the
hypercohomology of each complex of sheaves $\star$ \cite[Remark 2.1.6 (i)]{Dimca}. Hence the hypercohomology groups are
isomorphic.
\end{proof}

In general, one must be satisfied with the next corollary.

\begin{corollary}\label{Cor-KerCork}
Let $\Hyper^i$ denote the $i$th hypercohomology group. The kernel and cokernel of the map $\Hyper^i (\Omega^\bullet_\XW (\log \DW)) \rightarrow \Hyper^i (\Omega^\bullet_\XW ((k + \bullet) \DW))$ 
are killed by multiplication by $p^{\lfloor \log_p(k+1) \rfloor}$.
\end{corollary}

\begin{proof}
In this case by Corollary \ref{Cor-KillQ} for each $i$ the homology sheaves $\CurlyH^i(Q^\bullet)$ are killed by multiplication by
$p^{\lfloor \log_p(k+1) \rfloor}$ and hence so is the hypercohomology $\Hyper(Q^\bullet)$. The result follows from the long exact sequence in hypercohomology
for the short exact sequence (\ref{Eqn-Q}).
\end{proof}

\subsection{The hypercohomology of $\Omega^\bullet_\XW ((k + \bullet) \DW)$.}\label{Sec-HyperCoh}

The hypercohomology of $\Omega^\bullet_\XW((k + \bullet) \DW)$ is accessible because of the
following proposition.

\begin{proposition}\label{Prop-IfAcyclic}
Assume $k \geq 0$ is chosen so that the sheaves $\Omega^i_\XW ((k + i) \DW)$ are
acyclic for $i = 0,1,2$. Then the hypercohomology of $\Omega^\bullet_\XW((k + \bullet) \DW)$ is
canonically isomorphic to the homology of the complex  $\Gamma(\XW,\Omega^\bullet_\XW((k + \bullet) \DW))$ 
 of rings of global sections:
\[
0 \rightarrow \Gamma(\XW,\Oh_\XW ( k \DW))
\rightarrow \Gamma(\XW,\Omega_\XW ((k+1) \DW)) 
\stackrel{d}{\rightarrow} \Gamma(\XW, \Omega_\XW^2 ((k+2) \DW)) \rightarrow 0.
\]
\end{proposition}

\begin{proof}
The hypercohomology of $\Omega^\bullet_\XW((k + \bullet) \DW)$ may be calculated via a spectral
sequence, the first terms of which are the $\W$-modules $E_1^{p,q} :=
H^q(\XW,\Omega^p_\XW((k + p) \DW))$ \cite[Remark 2.1.6 (ii)]{Dimca}. Since the sheaves are acyclic, 
$E_1^{p,q} = 0$ for $q > 0$, and so this spectral sequence degenerates after the first term. Hence the hypercohomology is just the homology of the complex of rings of global sections.
\end{proof}

\begin{corollary}\label{Cor-NatMap}
Assume that $k \geq 0$ is chosen to satisfy the conditions of Proposition \ref{Prop-IfAcyclic}. Then
the natural map
\begin{equation}\label{Eqn-ExplicitLogDeRham}
 H^2_{cris}((\Xf,\Df)) \rightarrow \frac{\Gamma(\XW,\Omega^2_\XW((k + 2) \DW))}
{d(\Gamma(\XW,\Omega^1_\XW((k + 1) \DW))}
\end{equation}
has kernel and cokernel killed by multiplication by  $p^{\lfloor \log_p(k+1) \rfloor}$. In particular,
for $p > k + 1$ this map is an isomorphism.
\end{corollary}

\begin{proof}
This follows from Corollary \ref{Cor-KerCork} and Proposition \ref{Prop-IfAcyclic} and the natural
isomorphism $H_{dR}^\bullet ((\XW,\DW)) \cong H_{cris}^\bullet((\Xf,\Df))$.
\end{proof}

By Serre's vanishing theorem \cite[Theorem III.5.2.(b)]{Hart}, the hypothesis
of Proposition \ref{Prop-IfAcyclic} will be satisfied for $k \gg 0$.
It remains to calculate an explicit value for $k$ that one may choose. We shall do this for the
case in which $\XW$ is a smooth surface in $\pr^3$.

\begin{note}
The twisted logarithmic de Rham complex whose $i$th term is $\Omega^i_\XW(\log \DW) 
\otimes_{\Oh_\XW} \Oh_\XW(k \DW)$, which appears in \cite[Theorem 2.2.5]{AKR}, will
also be acyclic for $k \gg 0$. However, the complex $\Omega^\bullet_\XW((k + \bullet) \DW)$ seems
easier to compute with.
\end{note}

\subsection{Vanishing theorems for smooth surfaces in $\pr^3$}\label{Sec-Vanishing}

In this section we calculate explicit values of $k$ which make the sheaves
$\Omega^i_\XW((k + i) \DW)$ acyclic, for $\XW$ a smooth surface in $\pr^3_\W$.

Let $\XW$ be a smooth surface of degree $d$ in projective space $\pr := \pr^3_\W$ with
ideal sheaf $I \cong \Oh_{\spr} (-d)$. The latter condition is satisfied precisely when the surface
$\XW \otimes_\W \fq$ is of degree $d$. For $k \in \Z$ let $\Oh_\XW(k)$ denote the
twisting sheaf, and as usual define $\Omega^i_\XW(k) := \Omega^i_\XW \otimes_{\Oh_\XW} 
 \Oh_\XW(k)$.
Let $\DW$ be a hyperplane section. For any
$k \in \Z$ we have $\Omega^i_\XW(k \DW) := \Omega^i_\XW \otimes_{\Oh_\XW} \Oh_\XW(k \DW) \cong \Omega^i_\XW (k)$.

The proof of the following theorem was explained to the author by Bjorn Poonen.

\begin{theorem}[Poonen]\label{Thm-Poonen}
The sheaves $\Oh_\XW(k), \Omega_\XW(k)$ and $\Omega^2_\XW(k)$ are acyclic for $k > d-4$,\,
$k > \max\{2d - 4,0\}$, and $k  > 0$, respectively.
\end{theorem}

\begin{proof}
Let $F$ be the defining polynomial for the surface $\XW$, and $\pi: \XW \hookrightarrow \pr$.
The proof is based upon the following three exact sequences of sheaves c.f. \cite[II.8.17, II.8.13]{Hart}
\begin{equation}\label{S1}
0 \rightarrow \Oh_{\spr}(-d) \stackrel{\times F}{\rightarrow} \Oh_{\spr} \rightarrow \pi_* \Oh_\XW \rightarrow 0
\mbox{ on }\pr.
\end{equation}
\begin{equation}\label{S2}
0 \rightarrow I/I^2 \rightarrow \Omega_{\spr} |_\XW \rightarrow \Omega_\XW \rightarrow 0 
\mbox{ on } \XW 
\end{equation}
\begin{equation}\label{S3}
0 \rightarrow \Omega_{\spr} \rightarrow \Oh_\XW(-1)^{\oplus 4} \rightarrow \Oh_{\spr} \rightarrow 0
\mbox{ on } \pr,
\end{equation}
and the calculation of the homology of projective $3$-space $\pr$ c.f. \cite[III.5]{Hart}
\[
\begin{array}{l}
 H^0(\pr,\Oh_{\spr}(e)) = 0  \mbox{ for }e < 0\\
 H^1(\pr,\Oh_{\spr}(e)) = H^2(\pr,\Oh_{\spr}(e)) = 0  \mbox{ for all } e\\
 H^3(\pr,\Oh_{\spr}(e)) = 0  \mbox{ for }e > -4.
 \end{array}
 \]

 Since $\Oh_{\spr}(e)$ is locally free and hence flat we may twist (\ref{S1}) by $\Oh_{\spr}(e)$ and
 take homology to obtain the exact sequences
 \begin{equation}\label{S4}
   H^i(\pr,\Oh_{\spr}(e)) \rightarrow H^i(\pr, \pi_* \Oh_{\XW}(e)) \rightarrow
 H^{i+1}(\pr,\Oh_{\spr}(e-d)) \mbox{ for }i \geq 0.
 \end{equation}
 Since $\pi$ is affine, we have $H^i(\XW,\Oh_\XW(e)) \cong H^i(\pr,\pi_* \Oh_\XW(e))$ and
 from (\ref{S4}) and the homology of $\pr$ we see that the latter vanishes for $i > 0$ and $e - d > -4$.
 
 Next, restricting (\ref{S3}) to $\XW$, twisting by $\Oh_\XW(e)$ and taking homology we get exact sequences
 \[ H^{i-1}(\XW,\Oh_\XW(e)) \rightarrow H^i(\XW,\Omega_{\spr}(e)|_\XW) \rightarrow
 H^i(\XW,\Oh_\XW(e-1))^{\oplus 4} \mbox{ for }i \geq 1.\]
 Using our result on the vanishing of $H^i(\XW,\Oh_\XW(e))$ we immediately deduce
 \[ H^i (\XW,\Omega_{\spr}(e)|_\XW) = 0 \mbox{ for } i > 1 \mbox{ and } e > d-3.\]
 For $i = 1$, we need also determine when the the map
 \begin{equation}\label{S5}
  H^0(\XW,\Oh_\XW(e-1))^{\oplus 4} \rightarrow H^0(\XW,\Oh_\XW(e))
  \end{equation}
 is a surjection. 
  We need to determine explicitly what the spaces are and map is.
First, twisting (\ref{S1}) by $\Oh_{\spr}(e)$ and using the fact
 $H^1(\pr,\Oh_{\spr}(e)) = 0$ we get that  $H^0(\XW,\Oh_\XW(e))$ equals the quotient
of the space of homogeneous polynomials of degree $e$ modulo those of the form
$FG$ where $G$ is homogeneous of degree $e-d$. Second, observe that the map
in (\ref{S5}) comes from that in (\ref{S3}), and the $i$th component is just multiplication
by the $i$th variable $X_i$. So the map is surjective as soon as $e > 0$; for one can
obtain any monomial of degree $e > 0$ by multiplying some monomial of
degree $e - 1$ by one of the $X_i$. Thus the map is surjective when $e > 0$, and so
$H^1(\XW,\Omega_{\spr}(e)|_\XW) = 0$ for $e > \max\{0,d-3\}$.
 
Now observe since $I \cong \Oh_{\spr}(-d)$ we have
 \[ I/I^2 \cong I \otimes \Oh_{\spr}/I \cong \Oh_{\spr}(-d) \otimes \Oh_\XW = \Oh_\XW(-d).\]
So twisting (\ref{S2}) by $\Oh_\XW(e)$ and taking homology we get exact sequences
\[ H^i(\XW,\Omega_{\spr}(e)|_\XW) \rightarrow H^i(\XW,\Omega_\XW(e)) 
\rightarrow H^{i+1}(\XW,\Oh_\XW(e-d)) \mbox{ for } i  \geq 0.\]
 Thus $H^i(\XW,\Omega_\XW(e)) = 0$ for $i > 0$ and $e > \max\{0,2d-4\}$.
 
 Finally, taking top exterior powers in (\ref{S2}) one deduces $\Omega^2_\XW 
\cong \Oh_\XW(d-4)$, which completes the proof.
\end{proof}

\subsection{Crystalline lattices for smooth surfaces in $\pr^3$}\label{Sec-P3Lattice}

The next two theorems are the main results of Section \ref{Sec-CrysBases}. 

\begin{theorem}\label{Thm-ExpSurface}
Let $\Xf$ be a smooth surface in $\pr^3_\W$ of degree $d$ which remains of degree $d$ when
reduced modulo $p$, and $\Df$ the hyperplane section at infinity which we assume to be smooth. 
Let $k + 1 > \max\{2d-4,0\}$. Then the natural map
\begin{equation}\label{Eqn-ExpSurfaces}
 H^2_{cris}((\Xf,\Df)) \rightarrow \frac{\Gamma(\XW,\Omega^2_\XW((k + 2) \DW))}
{d(\Gamma(\XW,\Omega^1_\XW((k + 1) \DW)))}
\end{equation}
has kernel and cokernel killed by multiplication by  $p^{\lfloor \log_p(k+1) \rfloor}$.
\end{theorem}

\begin{proof}
Immediate from Corollary \ref{Cor-NatMap} and Theorem \ref{Thm-Poonen}
\end{proof}

\begin{theorem}\label{Thm-P3SurfaceLattice}
With notation as in Theorem \ref{Thm-ExpSurface}, let $\Hlogk$ denote the image of the righthand side
of (\ref{Eqn-ExpSurfaces}) in $H^2_{rig}(\Uf)$, where $\Uf = \Xf \backslash \Df$. Let
$\Hlog$ denote the image of $H^2_{cris}((\Xf,\Df))$ in $H^2_{rig}(\Uf)$. Then there is a natural
embedding $\Hlog \rightarrow \Hlogk$ with cokernel killed by multiplication by
$p^{\lfloor \log_p(k+1) \rfloor}$.
\end{theorem}

\begin{proof}
The maps $H^2_{cris}((\Xf,\Df)) \rightarrow \Hlog$ and 
\[ \frac{\Gamma(\XW,\Omega^2_\XW((k + 2) \DW))}
{d(\Gamma(\XW,\Omega^1_\XW((k + 1) \DW)))} \rightarrow \Hlogk \]
are isomorphisms modulo torsion. The map $\Hlog \rightarrow \Hlogk$ is defined via these isomorphisms, and
its cokernel is killed as claimed by Theorem \ref{Thm-ExpSurface}. It is
am embedding since the kernel of the map (\ref{Eqn-ExpSurfaces}) is torsion.
\end{proof}

\subsection{Frobenius-Hodge structures for smooth surfaces in $\pr^3$}\label{Sec-FHS}

In this section we gather some results which will be applied in Sections \ref{Sec-FHW} and \ref{Sec-Hodge}.
We begin with an entirely utilitarian definition whose relevance will shortly become apparent. Note that
it is exactly adequate for our purposes but inadequate for general use.

\begin{definition}\label{Def-FHS}
A torsion-free $F$-crystal
$H$ has a (pure) Frobenius-Hodge structure if there exists $j,m \in \Z$ with $j,m \geq 0$ and a 
filtration by submodules $0 =  H_{m+1} \subseteq H_{m} \subseteq \cdots \subseteq H_0 = H$ such that
$F(H_i) \subseteq p^{j+i}(H)$ for $0 \leq i \leq m$. A torsion-free $F$-crystal $H$ has a (mixed) Frobenius-Hodge
structure if there exists a short exact sequence of $F$-crystals $0 \rightarrow H_1 \rightarrow H \rightarrow H_2
\rightarrow 0$ such that $H_1$ and $H_2$ have pure Frobenius-Hodge structures.
\end{definition}

Let $\XW$ be a smooth hypersurface in $\pr^n_\W$ with special fibre $\Xf$ and generic
fibre $\XK$. Then the $\W$-modules
$H^j(\XW,\Omega^i_{\XW})$ are torsion free \cite[Page 660, Lines 15-16]{BM}, and thus so 
are $H^m_{cris}(\Xf)$ (this is implied by the remarks \cite[Page 662, Lines 7-9]{BM} and the statement
of \cite[Theorem 2]{BM}).
We define the
Hodge numbers $h^{i,j} := \rk(H^j(\XW,\Omega^i_{\XW})) = 
\dim (H^j(\XK,\Omega^i_{\XK}))$. Note that $h^{i,j} = h^{j,i}$. Let $p$ be the residue characteristic
of $\W$, and assume that $p > n-1 = \dim(\Xf)$. The following result is from \cite[Section 5]{BM}.
There exist a filtration (Hodge filtration)
\[ 0 := H_{m+1,m} \subseteq H_{m,m} \subseteq H_{m-1,m} \subseteq \cdots \subseteq H_{0,m} = H^m_{cris}(\Xf)\]
such that the $p$th power Frobenius $F$ acts as $F(H_{i,m}) \subseteq p^i H^m_{cris}(\Xf)$. Moreover,
$\rk (H_{i,m}/H_{i+1,m}) = h^{i,m-i}$ for $0 \leq i \leq m$. Thus
$H^m_{cris}(\Xf)$ has a pure Frobenius-Hodge structure, see Definition \ref{Def-FHS}, with
the ranks of the submodules in the filtration determined by the Hodge numbers of $\XK$.

Now let $\XW$ be a smooth surface in $\pr^3_{\W}$. Let $\DW \subset \pr^2_\W$ be the curve at infinity, which we assume to be
smooth. Then the results in the above paragraph apply to give pure Frobenius-Hodge structures on $H^2_{cris}(\XW)$ and
$H^1_{cris}(\DW)$. Since $H^1_{cris}(\Xf) = H^3_{cris}(\Xf)$ are torsion-free, and we know $H^1_{rig}(\Xf) = H^3_{rig}(\Xf) = 0$,
by the Lefschetz hyperplane theorem, one deduces $H^1_{cris}(\Xf) = H^3_{cris}(\Xf)= 0$. Moreover, since the cycle
class map is injective in codimension $1$ \cite[Page 97]{JT0}, from (\ref{Seq-Loc1}) one deduces the following exact sequence.
\begin{equation}\label{Eqn-MFHS}
 0 \rightarrow \W(-1) \rightarrow H^2_{cris}(\Xf) \rightarrow H^2_{cris}((\Xf,\Df)) \rightarrow H^1_{cris}(\Df)(-1)\rightarrow 0.
\end{equation}
Note that $F$ acts by multiplication by $p$ on the rank $1$ $F$-crystal $\W(-1)$, so $H^2_{cris}(\Xf)/\W(-1)$
has a pure Frobenius-Hodge structure.
The sequence (\ref{Eqn-MFHS}) shows that $H^2_{cris}((\Xf,\Df))$ is torsion-free, and defines a mixed Frobenius-Hodge structure on $H^2_{cris}((\Xf,\Df))$, see Definition \ref{Def-FHS}

\section{Weighted projective surfaces}\label{Sec-deJong}

In this section we consider the case in which $\XW$ is a surface of degree $d$ in the weighted projective
space $\pr(1,a,b,c)_\W$. Following a suggestion of A.J. de Jong, we view $\XW$ in a standard way as the quotient
of a surface $\YW$ in $\pr^3$ by the action of a finite group $G \cong
\Z/(a) \times \Z/(b) \times \Z/(c)$. We assume that $\YW$ is smooth with the $G$-invariant divisor at infinity $\EW$ smooth, and also
that the affine piece $\UW := \XW \cap \Aff^3_\W$ is smooth. Under these assumptions,
we define a lattice $\Hlog \subset H^2_{rig}(\Uf)$ and for each $k \geq \max\{2d-4,0\}$ explicitly construct a lattice
$\Hlogk \subset H^2_{rig}(\Uf)$ such that there is a embedding $\Hlog \rightarrow \Hlogk$ with
cokernel killed by multiplication by $p^{\lfloor \log_q(k+1) \rfloor}$. We explain the importance of the lattices
$\Hlog$ and $\Hlogk$ in Section \ref{Sec-Hodge}.

\subsection{The lattices $\Hlog \subseteq \Hlogk$ in $H^2_{rig}(\Uf)$}\label{Sec-LatticeWPS}

We first define our ambient weighted projective space.

\begin{definition}
Let $a,b,c$ be positive integers, and
$\pr(1,a,b,c)_\W$ be a weighted projective space over $\W$; that is, 
$\Proj(S)$ where the ring $S := \W[T,X,Y,Z]$ is graded by assigning weights $\wt(T) :=1, \wt(X):=a, \wt(Y) :=b, \wt(Z) := c$. 
Assume that the residue characteristic $p$ of $\W$ does not divide $abc$.
\end{definition}

Note that if desired one can assume that $\gcd(a,b,c) = 1$ (although this plays no part in our proofs) since
in any case $\pr(1,a,b,c)_\W \cong \pr(1,a/e,b/e,c/e)_\W$ where $e := \gcd(a,b,c)$ \cite[Page 186]{DimcaCrelle}.

The
open subscheme defined by $T \ne 0$ is isomorphic to affine $3$-space
$\Aff^3_\W$ with coordinates given by $x := X/T^a, y:= Y/T^b$ and $z := Z/T^c$.
Let $R:= \W[X_0,X_1,X_2,X_3]$ be the graded polynomial ring with $\wt(X_i) := 1$ for 
$0 \leq i \leq 3$ and define projective $3$-space $\pr^3_\W := \Proj(R)$. Consider the injective map of graded rings
\[ \pi: S \rightarrow R ,\, T \mapsto X_0,\,
X \mapsto X_1^a,\, Y \mapsto X_2^b,\, Z \mapsto X_3^c.\]
This induces a quotient map $\pr_\W^3 \rightarrow \pr(1,a,b,c)_\W$ of topological spaces.

Assume now that $\W$ is large enough to contain primitive $a$th, $b$th and $c$th roots of
unity $\zeta_a$, $\zeta_b$ and $\zeta_c$, respectively. See Note \ref{Note-NoPrimRoots} for a
discussion of the general case.
Let the group $G := \Z/(a) \times \Z/(b) \times \Z/(c)$ act on
$R$ by
\[ G \times R \rightarrow R,\,
(i,j,k) \times f(X_0,X_1,X_2,X_3) \mapsto f(X_0,\zeta_a^i X_1,\zeta_b^j X_2,\zeta_c^k X_3).\]

\begin{lemma}\label{Lem-InvariantRing}
The invariant subring $R^G$ of $R$ under the action of $G$ is exactly $\pi(S) = R[X_0,X_1^a,X_2^b,X_3^c]$.
\end{lemma}

\begin{proof}
Certainly $\pi(S) \subseteq R^G$. Suppose now that $f  = \sum_{i} a_i X^i \in R^G$ where
$i := (i_0,i_1,i_2,i_3) \in \Z_{\geq 0}^4$ and $X^i := X_0^{i_0} \cdots X_3^{i_3}$. Since $G$ acts on each term 
$a_i X^i$ in $f$ by scalar multiplication, each term in $f$ must be fixed by $G$. So it is enough to consider the
case of a single monomial $f = X_0^{i_0} \cdots X_3^{i_3}$. By considering the action of $(1,0,0),(0,1,0),(0,0,1) \in G$
in turn, one deduces that $a|i_1,b|i_2$ and $c|i_3$, so $f \in \pi(S)$, as required.
\end{proof}

\begin{corollary}
The weighted projective space $\pr(1,a,b,c)_\W$ is isomorphic to the quotient scheme $\pr^3_\W/G$.
\end{corollary}

\begin{proof}
By Lemma \ref{Lem-InvariantRing}, $S \stackrel{\pi}{\cong} R^G$, and so $\Proj(S) \cong
\Proj(R^G) =: \pr^3_\W/G$.
\end{proof}

Let $\XW$ be a hypersurface in $\pr(1,a,b,c)_\W$ defined by a weighted homogeneous polynomial
$P(T,X,Y,Z)$ of degree $d$. Assume that the hypersurface $\YW$ in $\pr^3_\W$ defined by the
polynomial $\pi(P) = P(X_0,X_1^a,X_2^b,X_3^c)$ is a smooth $\W$-scheme with ideal sheaf $I \cong \Oh_{{\spr}^3_\W}(-d)$.
Define
$\DW$ and $\EW$ to be the divisors on $\XW$ and $\YW$, respectively, defined by the
ideals $(T)$ and $(X_0)$, respectively. Let $\UW := \XW \backslash \DW$ and
$\UWtil := \YW \backslash \EW$. 

\begin{lemma}
The group $G$ acts on $\YW$ with $\XW \cong \YW/G$, on $\UWtil$ with $\UW \cong \UWtil/G$, and
on $\EW$ with $\DW \cong \EW/G$.
\end{lemma}

\begin{proof}
The group $G$ acts on $\YW$ since the ideal $(\pi(P))$ is $G$-invariant. Now $\YW/G := \Proj((R/(\pi(P)))^G)$ and we note
$(R/(\pi(P)))^G \cong R^G/(\pi(P))$. 
The isomorphism $S \stackrel{\pi}{\rightarrow} R^G$ reduces modulo the homogeneous ideal $(P)$ to an isomorphism
$S/(P) \stackrel{\pi}{\rightarrow} R^G/(\pi(P)) = (R/(\pi(P)))^G$. So $\XW = \Proj(S/(P)) = \YW/G$.

Since $G$ acts trivially on $T$, it acts on $\UW$ with $(i,j,k) \times g(x,y,z) \mapsto g(\zeta_a^i x, \zeta_b^j y,\zeta_c^k z)$, for
$g \in \W[x,y,z] = \W[X/T^a,Y/T^b,Z/T^c]$. The proof that $\UW \cong \UWtil/G$ is similar to that in the preceding
paragraph. The group $G$ acts on $\EW$ since it fixes the homogeneous ideal $(T)$, and one sees as before
that $\DW \cong \EW/G$.
\end{proof}

Let $\Xf,\Df,\Uf,\Yf,\Ef,\Uftil$ denote the special fibres. Assume that $\UW$ is smooth.

\begin{lemma}\label{Lem-GonRig}
For each $i \geq 0$, the group $G$ acts on the rigid cohomology space $H^i_{rig}(\Uftil)$ and $H^i_{rig}(\Uf) \cong H^i_{rig}(\Uftil)^G$.
\end{lemma}

\begin{proof}
Let $A$ and $B$ denote the coordinate rings of $\UW$ and $\UWtil$, respectively; so
$A \cong B^G$. Taking weak completions we have $A^\dagger \cong (B^G)^\dagger \cong (B^\dagger)^G$, the latter
isomorphism follows since the action of $G$ commutes with taking weak completions.
Let $d: B^\dagger \rightarrow \Omega_{B^\dagger}$ be the universal derivation which is
continuous w.r.t the $p$-adic metric on $B^\dagger$. One checks that
$(\Omega_{B^\dagger}^\bullet)^G$ is
the direct factor in $\Omega_{B^\dagger}^\bullet$ corresponding to the trivial character in a character-based
decomposition, that
$(\Omega_{B^\dagger}^\bullet)^G \cong \Omega_{(B^\dagger)^G}^\bullet$ as complexes
of $(B^\dagger)^G$-modules, and since $(B^\dagger)^G \cong A^\dagger$ we have 
$\Omega_{(B^\dagger)^G}^\bullet \cong \Omega_{A^\dagger}^\bullet$.
The rigid cohomology of $\Uftil$ is by definition the homology $H(\Omega_{B^\dagger}^\bullet)$. Since
$(\Omega_{B^\dagger}^\bullet)^G$ is the trivial-character direct factor in $\Omega^\bullet_{B^\dagger}$ one has
$H(\Omega_{B^\dagger}^\bullet)^G \cong H((\Omega_{B^\dagger}^\bullet)^G)$, and
we know $H((\Omega_{B^\dagger}^\bullet)^G) \cong
H(\Omega_{A^\dagger}^\bullet)$. But the latter is by definition the rigid cohomology of $\Uf$.
\end{proof}

Assume that $\EW$ is a smooth divisor. Then the crystalline cohomology groups $H^i_{cris}((\Yf,\Ef))$ are defined, and
there is a natural map
\[ H^i_{cris}((\Yf,\Ef)) \rightarrow H^i_{rig}(\Uftil)\]
which is an injection modulo torsion with image a full-rank $\W$-lattice. These morphisms commute with the action of
$G$, so we may take $G$-invariants to get natural maps
\[ H^i_{cris}((\Yf,\Ef))^G \rightarrow H^i_{rig}(\Uftil)^G \]
and by Lemma \ref{Lem-GonRig} the image is isomorphic to $H^i_{rig}(\Uf)$.

\begin{definition}\label{Def-LogCris}
For each $i \geq 0$, define the $\W$-module $H^i_{cris}((\Xf,\Df))$ to be the $G$-invariant part $H^i_{cris}((\Yf,\Ef))^G$. We shall
call these groups the log-crystalline cohomology of the pair $(\Xf,\Df)$.
\end{definition}

\begin{proposition}\label{Prop-Lattice}
There
is a natural map which is an injection modulo torsion $H^i_{cris}((\Xf,\Df)) \rightarrow H^i_{rig}(\Uf)$, with image
a full-rank $F$-invariant $\W$-lattice. Let us denote this lattice $\Hlog$.
\end{proposition}

\begin{proof}
We constructed the natural map immediately before Definition \ref{Def-LogCris}, and it is an injection modulo
torsion. It is $F$-invariant since the action of $G$ commutes with the $p$th power map on the homogeneous coordinate ring $R/(\pi(P),p)$
of $\Yf$.
Since the order of $G$ is coprime to $p$, there is a $G$-invariant decomposition
\[ H^i_{cris}((\Yf,\Ef)) = H^i_{cris}((\Yf,\Ef))^G \oplus C_{cris} \rightarrow H^i_{rig}(\Uftil)^G \oplus C_{rig} = H^i_{rig}(\Uftil)\]
and a natural map which is an injection modulo torsion $C_{cris} \rightarrow C_{rig}$. If the map
in the proposition did not have a full-rank image, then the rank of the image of $C_{cris}$ in $C_{rig}$ would have
to exceed the dimension of $C_{rig}$; an impossibility.
\end{proof}

\begin{note}\label{Note-IsLC}
The $\W$-scheme $\XW$ in general will only be quasi-smooth, that is have quotient singularities, which in our
case lie along the divisor
$\DW$ \cite[Appendix B]{DimcaW}.
So the log-crystalline cohomology $H^i_{cris}((\Xf,\Df))$ is not a priori defined. The divisor $\DW$ itself is a quasi-smooth
curve in weighted projective space $\pr(a,b,c)_\W$. Quasi-smooth varieties are normal, and hence the curve $\DW$
itself is smooth. The author believes that $H^2_{cris}((\Xf,\Df))$ defined above is isomorphic to
$H^2_{cris}((\Xf_1,\Df_1))$ where $\Xf_1$ is some smooth compactification of $\Uf$ with $\Df_1 := \Xf_1 \backslash
\Uf_1$ a smooth normal crossing divisor; however, he offers no proof of this. 
\end{note}

Using Theorem \ref{Thm-ExpSurface}, the group $H^2_{cris}((\Yf,\Ef))$ can be explicitly described as the cohomology classes
in $H^2_{rig}(\Uftil)$ of a set of elements in $\Gamma(\UWtil,\Omega_{\UWtil}^2)$.
  The $G$-invariants of this
set are easily calculated. Explicitly, for $k+1 > \max\{2d-4,0\}$ we see that the map
\begin{equation}\label{Eqn-QuasiSmoothLattice}
 H^2_{cris}((\Xf,\Df)) :=
 H^2_{cris}((\Yf,\Ef))^G \rightarrow \frac{\Gamma(\YW,\Omega^2_{\YW}((k+2) \EW))^G}
{d(\Gamma(\YW,\Omega^1_{\YW}((k + 1) \EW)))^G} = 
 \frac{\Gamma(\XW,\Omega^2_\XW((k+2) \DW))}
{d(\Gamma(\XW,\Omega^1_\XW((k + 1) \DW)))}
\end{equation}
has kernel and cokernel killed by multiplication by $p^{\lfloor \log_p(k+1) \rfloor}$.
The last equality of $\W$-modules in (\ref{Eqn-QuasiSmoothLattice}) follows from the isomorphism
$\Gamma(\UW,\Omega_{\UW}^i) \cong \Gamma(\UW,(\rho_* \Omega_{\UWtil}^i)^G)$ where
$\rho: \UWtil \rightarrow \UW$ is the quotient map, c.f. \cite[Page 215, Lines 8-12]{JS}, 
and the fact that
pole orders are just calculated by taking suitable weighted degrees.

\begin{definition}\label{Def-Hlogk}
For each $k +1 \geq \max\{2d-4,0\}$, the lattice $\Hlogk$ is defined to be the image of the righthand side of
(\ref{Eqn-QuasiSmoothLattice}) in $H^2_{rig}(\Uf)$.
\end{definition}

\begin{proposition}
There is a natural embedding
$\Hlog \rightarrow \Hlogk$ with cokernel killed by multiplication by $p^{\lfloor \log_p(k+1) \rfloor}$.
\end{proposition}

\begin{proof}
Elements in the kernel of the natural map (\ref{Eqn-QuasiSmoothLattice}) must be
torsion (take $G$-invariants in Corollary \ref{Cor-KerCork} applied to the pair $(\Yf,\Ef)$) and so vanish in $\Hlog = {\rm Im}(H^2_{cris}((\Xf,\Df)) \rightarrow H^2_{rig}(\Uf))$.
The map in proposition is thus the embedding defined by pulling back elements
in $\Hlog$ to the free part of $H^2_{cris}((\Xf,\Df))$, mapping via (\ref{Eqn-QuasiSmoothLattice}),
and then embedding in $H^2_{rig}(\Uf)$.
\end{proof}

The $K$-vector space $\Hlogk \otimes_\W \K$ is $F$-invariant and so has the structure of an
$F$-isocrystal; however, $\Hlogk$ itself is not in general $F$-invariant. 

\begin{note}\label{Note-NoPrimRoots}
Only the construction of the lattice $\Hlog$ 
relies upon the existence of appropriate primitive roots $\zeta_a,\zeta_b,\zeta_c \in \W^*$.
The definition of $\Hlogk$ does not, and indeed this lattice is independent of the choice of base field --- see
Section \ref{Sec-Surfaces} for an example of how to construct this lattice.
In the general case, one constructs a basis for the lattice
$\Hlogk$. The theory
guarantees that if one extends the base field $\fq$ to contain appropriate primitive roots, then this lattice is
close to $\Hlog$, and so is
acted upon in a particular way by the $p$th power Frobenius map $F$, see Section \ref{Sec-Hodge}. But the action of the $p$th power Frobenius
map is independent of the field over which the basis is defined, so it acts in that manner on the original lattice $\Hlogk
\subset H^2_{rig}(\Uf)$.
\end{note}

\begin{note}
A direct application of Poonen's method to surfaces in weighted projective space and experiments suggest that in fact $k+1 > \max\{2d-(1 + a + b + c),0\}$ is sufficient for the map in (\ref{Eqn-QuasiSmoothLattice}) to have kernel and
cokernel killed as claimed.
\end{note}

\subsection{Frobenius-Hodge structures for weight projective surfaces}\label{Sec-FHW}

We continue with the notation and assumptions in Section \ref{Sec-LatticeWPS}. We shall also
use the notation in the final paragraph of Section \ref{Sec-FHS}, only we now
denote our smooth projective surface and curve as $\YW$ and $\EW$, respectively. Taking $G$-invariants in the
exact sequence (\ref{Eqn-MFHS}) for the pair $(\Yf,\Ef)$
and noting the first non-zero term is $H^0_{cris}(\Ef)$ we get the sequence
\[ 0 \rightarrow \W(-1) \rightarrow H^2_{cris}(\Yf)^G \rightarrow H^2_{cris}((\Xf,\Df)) \rightarrow H^1_{cris}(\Ef)(-1)^G \rightarrow 0. \]
Since Frobenius $F$ commutes with the action of $G$, the above exact sequence defines a
mixed Frobenius-Hodge structure (Definition \ref{Def-FHS}) on $H^2_{cris}((\Xf,\Df))$. We examine this in more detail next.

Let us first recall how one defines Hodge numbers of quasi-smooth varieties, c.f. \cite[Appendix A.3]{CK}.

\begin{definition}\label{Def-HodgeNos}
Let $j:\XK^s \rightarrow \XK$ be the embedding of the smooth locus and $i \geq 0$, and 
define the sheaf $\hat{\Omega}_{\XK}^i := j_* \Omega_{\XK^s}^i$ on $\XK$. The Hodge numbers $h^{i,j}$ are
the dimensions of the $\K$-vector spaces $H^j(\XK,\hat{\Omega}_{\XK}^i)$. 
\end{definition}

Note that $\DK$ is smooth so its Hodge numbers are defined
in the usual manner, as the dimensions of the spaces $H^j(\DK,\Omega_{\DK}^i)$.
By the result on \cite[Page 215, Lines 8-12]{JS}, for each $i \geq 0$ we have an isomorphism of sheaves
$(\rho_* \Omega_{\EK}^i)^G \cong \Omega^i_{\DK}$ and $(\rho_* \Omega_{\YK}^i)^G \cong \hat{\Omega}^i_{\XK}$ where $\rho :\YK \rightarrow \XK$ is the quotient map and also its restriction $\rho: \EK \rightarrow \DK$. It follows
by decomposing sheaves via characters of $G$ and using the finiteness of $\rho$ 
that $H^j(\YK,\Omega^i_{\YK})^G \cong H^j(\XK,\hat{\Omega}^i_{\XK})$ and
$H^j(\EK,\Omega^i_{\EK})^G \cong H^j(\DK,\Omega^i_{\DK})$.
Thus the dimensions of the graded pieces in the filtrations on
$H^2_{cris}(\Yf)^G$ and $H^1_{cris}(\Ef)^G$ are just the Hodge numbers of the quasi-smooth
surface $\XK$ and smooth curve $\DK$. 

Since $H^2_{cris}((\Xf,\Df))$ is a submodule of
$H^2_{cris}((\Yf,\Ef))$ and the latter is torsion free, so is the former. Thus the lattice
$\Hlog := {\rm Im}(H^2_{cris}((\Xf,\Df)) \rightarrow H^2_{rig}(\Uf))$ is isomorphic to
$H^2_{cris}((\Xf,\Df))$. Defining $\Hprim := H^2_{cris}(\Xf)^G/\W(-1)$ we see that
there is a filtration  $0 \subseteq H_2 \subseteq H_1 \subseteq H_0 = \Hprim$ with
 $F(H_i) \subseteq p^i \Hprim$. Recall $h^{i,j} := \dim (H^j (\XK,\hat{\Omega}_{\XK}^i))$ are the
 Hodge numbers of the quasi-smooth variety $\XK$. We have $h^{2,0} = \rk H_2$ and
 $h^{1,1} - 1 = \rk(H_1/H_2)$ (we have removed the hyperplane class), and $h^{0,2} = \rk (H_0/H_1)$.
 Defining $h_2 := h^{2,0} + h^{1,1} + h^{0,2}$, the middle
 Betti number of $\XK$, we see that $\Hprim$ has rank $h_2 - 1$.
 
 Likewise, the lattice
 \[ H(\Df) := {\rm Im}(H^1_{cris}(\Ef)^G \rightarrow H^1_{rig}(\Ef)^G)\]
 is isomorphic to $H^1_{cris}(\Ef)^G$ and
 has a filtration $0 \subseteq H_1^\prime \subseteq H_0^\prime = H(\Df)$ with $F(H_i^\prime) \subseteq p^i (H(\Df))$ and $\rk H_1^\prime = g =
 \rk (H_0^\prime/H_1^\prime)$ where $g := \dim(H^0(\DK,\Omega^1_{\DK}))$. So the mixed 
 Frobenius-Hodge structure on $\Hlog \cong H^2_{cris}((\Xf,\Df))$ comes from the exact
 sequence
 \[ 0 \rightarrow \Hprim \rightarrow \Hlog \rightarrow H(\Df)(-1) \rightarrow 0 \]
 and the dimensions of the various graded pieces can be calculated from the Hodge numbers
 of the quasi-smooth surface $\XK$ and smooth curve $\DK$.

\subsection{Surfaces of the form $z^2 = Q(x,y)$}\label{Sec-Surfaces}

We now describe an explicit method for constructing the lattice
$\Hlogk$ in the case in which $\UW$ is a surface defined by a polynomial $z^2 = Q(x,y)$.

\subsubsection{Differential forms on the surface}

Assume that the characteristic $p$ of the residue field of $\W$ is odd.
Let $a,b,c$ be positive integers with $\gcd(a,b) = 1$ and $abc$ not divisible by $p$. Let
$Q(x,y) = \sum_{i,j} q_{i,j} x^i y^j$ be a polynomial in $\W[x,y]$ such that
$ai + bj \leq 2c$ for all pairs $(i,j)$, with this bound met for some choice of
pair $(i,j)$; that is, $Q(x,y)$ has weighted degree $2c$ when one assigns
weights $\wt(x) := a$ and $\wt(y) := b$ in the polynomial ring $\W[x,y]$. Let
$\XW$ be the hypersurface of degree $d := 2c$ in $\pr(1,a,b,c)_\W$ defined by the polynomial
\[ P := T^d \left((Z/T^c)^2 - Q(X/T^a,Y/T^b) \right).\]
Let $\DW$ denote the scheme defined by the ideal $(P,X_0)$ and $\UW := \XW \backslash \DW$. The
affine $\W$-scheme $\UW$ is defined by the polynomial $z^2 = Q(x,y)$, where
$x  = X/T^a, y = Y/T^b$ and $z := Z/T^c$. We assume that $\UW$ is smooth. 
We further assume that the schemes $\YW$ and $\EW$, respectively, in $\pr_\W^3$ defined by the polynomial
$\pi(P) := P(X_0,X_1^a,X_2^b,X_3^c)$ and ideal $(\pi(P),X_0)$, respectively, are smooth $\W$-schemes, although
this will play no role in the construction of the lattice $\Hlogk$.

\begin{lemma}\label{Lem-PoleOrder}
The functions $x,y,$ and $z$ have poles of orders $a,b$ and $c$, respectively, along the divisor $\DW$. The
$1$-forms $dx,dy$ and $dz$ have poles of orders $a+1,b+1$ and $c+1$ along $\DW$, and the
$2$-form $dx \wedge dy$ a pole of order $a + b + 2$ along $\DW$.
\end{lemma}

\begin{proof}
Follows since
the divisor $\DW$ is defined by the homogeneous ideal $(T)$, differentiation increases pole orders by one, and
$dx \wedge dy = d (y  dx)$.
\end{proof}

By definition, the module $\Gamma(\UW,\Omega^1_\UW)$ is generated over
$\W[x,y,z]/(z^2 - Q(x,y))$ by the $1$-forms $dx$, $dy$ and $dz$ subject to the one relation
\[ 2z dz = Q_x dx + Q_y dy.\]
It follows immediately that the module 
$\Gamma(\UW,\Omega^2_\UW)$ is generated over $\W[x,y,z]/(z^2 - Q(x,y))$ by the
$2$-forms $dx \wedge dy, dy \wedge dz$ and $dz \wedge dx$ subject to the relations
\[ dy \wedge dz = - \frac{Q_x}{2z} dx \wedge dy,\, dz \wedge dx = - \frac{Q_y}{2z} dx \wedge dy.\]
Since $\UW$ is smooth we have $(Q,Q_x,Q_y) = (1)$. A simple computation now reveals that 
\[ \Gamma(\UW,\Omega^2_\UW) = \left\{ \frac{A + B z}{z} dx \wedge dy\,|\, A,B \in \W[x,y]\right\}.\]
By Lemma \ref{Lem-PoleOrder}, the pole order of such a $2$-form along $\DW$ is bounded above by
\[ \max\{\wt(A)-c,\wt(B)\} + a + b + 2\]
where $\wt$ here means weighted degree. 

\begin{lemma}\label{Lem-ZeroClass}
Let $B \in \W[x,y]$ have ordinary degree $\deg(B)$. Then for any $k +2 \geq \wt(B) + a + b + 2$,
the class of the $2$-form
$B dx \wedge dy$ in the quotient on the righthand side of (\ref{Eqn-QuasiSmoothLattice}) is killed
by multiplication by $p^{\lfloor \log_p(\deg(B) \rfloor}$; thus it is the zero class in $H^2_{rig}(\Uf)$.
\end{lemma}

\begin{proof}
Integrating w.r.t. $x$, say, one constructs a polynomial $C \in \W[x,y]$ with
$dC/dx = p^{\lfloor \log_p(\deg(B) \rfloor} B$. Then $d(C \wedge dy) = p^{\lfloor \log_p(\deg(B) \rfloor} 
B dx \wedge dy$, as required.
\end{proof}

Since we wish to construct the image in $H^2_{rig}(\Uf)$ of the righthand side of (\ref{Eqn-QuasiSmoothLattice}),
by Lemma \ref{Lem-ZeroClass}
we can restrict attention to $2$-forms $(A/z) dx \wedge dy$.

By an explicit computation one finds that the exact forms $(A/z) dx \wedge dy$ are precisely those for which
\[ 2 A = \alpha Q_x + \beta Q_y + 2(\alpha_x + \beta_y) Q,\, \alpha,\,\beta \in \W[x,y].\]
Precisely, such a form is $d$ of
\[ C z dx + D z dy + E dz,\,\alpha :=D - E_y,\,\beta:=-C + E_x,\, C,D,E \in \W[x,y].\] 
Note that the pole order of such a $1$-form is bounded by $\max\{\wt(\alpha)  + c + b + 1,\wt(\beta) + c + a + 1\}$.
So to ensure this has pole order at most $k+1$ we need to have $\wt(\alpha) \leq
(k+1) - (c + b + 1) = k - (b+c)$ and $\wt(\beta) \leq k - (a+c)$.
Moreover,  the $2$-forms
$(A/z) dx \wedge dy$ have pole order bounded by $\wt(A) - c + a + b + 2$, so to ensure
the pole order is at most $k + 2$ it is enough to have $\wt(A) \leq (k+2) + c - (a + b + 2) = k + c - (a+b)$.

\subsubsection{An algorithm for constructing the lattice $\Hlogk$}

For any $k \geq 0$, one constructs a set which generates the quotient in (\ref{Eqn-QuasiSmoothLattice}) as follows. For any
integer $m \geq 0$, let ${\rm Mon}_{\leq m}$ denote the union of the set of monomials $x^i y^j$ of
weighted degree at most $m$ with $\{0\}$. Let $Z_{m}$ denote the free
$\W$-submodule of $\W[x,y]$ spanned by the monomials in ${\rm Mon}_{\leq m + c - (a + b + 2)}$.
Let $B_{m} \subseteq Z_{m}$ be the $\W$-module generated by polynomials of the form
$\alpha Q_x + \beta Q_y + 2 (\alpha_x + \beta_y)$ where
$\alpha \in {\rm Mon}_{\leq (m-1) - (b+c+1)}$ and $\beta \in {\rm Mon}_{\leq (m-1) - (a+c+1)}$.
Define $H_{m} := Z_{m}/B_{m}$. Then $H_{k+2}$ is isomorphic to the quotient  in
(\ref{Eqn-QuasiSmoothLattice}) via the map
\[ \theta: Z_{k+2} \rightarrow \Gamma(\XW,\Omega_{\XW}^2((k+2)\DW)),\, x^i y^j \mapsto x^i y^j dx \wedge dy/2z.\]
 One can find a set
$S(k) \subseteq {\rm Mon}_{\leq k+c - (a+b)}$ of monomials whose classes in $H_{k+2}$ are a $\W$-generating set for this quotient using the  following method. 

Let $Z_{m,p}:= Z_m \otimes_\W \fq$ and $B_{m,p} := B_m \otimes_\W \fq$ be $\fq$-vector spaces. Using linear algebra over $\fq$, compute a generating set $S(k)$ of monomials for the
quotient $Z_{k+2,p}/B_{k+2,p}$. Let $\langle S(k) \rangle_\W$ denote the $\W$-span of this
generating set.
Let $w \in Z_{k+2}$. Then $w = u_1 + v_1 \bmod{p}$ for some
$u_1 \in \langle S(k) \rangle_\W$ and $v_1 \in B_{k+2}$. Write $w = 
u_1 + v_1 + pw_2$ where $w_2 \in Z_{k+2}$. Repeating the process we find $u_2
 \in \langle S(k) \rangle_\W$ and $v_2 \in B_{k+2}$ such that
$w = (u_1 + pu_2) + (v_1 + pv_2) + p^2 w_2$ for some $w_2 \in Z_{k+2}$. By induction, and taking
a $p$-adic limit, we find $u \in \langle S(k) \rangle_\W$ and $v \in B_{k+2}$ such that
$w = u + v$, as required. 

If $p > k + 1$ and $k \geq \max\{2d - 4,0\}$, and the classes of the forms
\begin{equation}\label{Eqn-forms}
 \theta(s)  = (s/2z) dx \wedge dy \mbox{ for }s \in S(k)
\end{equation}
are independent in $H^2_{rig}(\Uf)$, then the forms (\ref{Eqn-forms}) are a basis for the lattice
\[ \Hlog  = {\rm Im}(H^2_{cris}((\Xf,\Df)) \rightarrow H^2_{rig}(U)).\]
Note that they will be independent in $H^2_{rig}(\Uf)$ if and only if $|S(k)| = \dim H^2_{rig}(\Uf)$; the latter
is just $(\dim H^2_{rig}(\Xf) - 1) + \dim H^1_{rig}(\Df)$ (c.f. first paragraph of the proof of
Proposition \ref{Prop-Inj}), and so is easily calculated.

If $k \geq \max\{2d-4,0\}$
and the classes of (\ref{Eqn-forms}) are independent in $H^2_{rig}(\Uf)$, then the classes
of (\ref{Eqn-forms}) span a lattice $\Hlogk \subset H^2_{rig}(\Uf)$ such that there is a natural embedding
$\rho: \Hlog \rightarrow \Hlogk$ with cokernel killed
by $p^{\lfloor \log_p (k+1) \rfloor}$. If the classes are not independent in $H^2_{rig}(\Uf)$ then
one must do some further work to extract a basis for the free part of $\langle S(k) \rangle_\W$. We have not come across this situation in practice 
and so do not address it here.

\section{Frobenius-Hodge structures and precision estimates}\label{Sec-Hodge}

In this section we explain how the use of a basis for the full-rank $\W$-lattice 
$\Hlogk$ in $H^2_{rig}(\Uf)$ when calculating the Frobenius action allows
one to compute with lower $p$-adic precisions. 

Notation and assumptions are as in Section \ref{Sec-deJong}. In particular, $\XW := \YW/G$ where $G$ is a finite group of order
prime to $p$ and
$\YW$ a smooth surface in $\pr^3_{\W}$, $\EW$ is the divisor at infinity of
$\YW$, assumed smooth and fixed by $G$, and $\UW := \UWtil/G$ (where 
$\UWtil := \YW \backslash \EW$) is smooth. We denote by $\Xf,\Df$ and $\Uf$ the special
fibres of $\XW$, $\DW$ and $\UW$, respectively. For $k \geq \max\{2d-4,0\}$ the $\W$-lattice
$\Hlogk \subset H^2_{rig}(\Uf)$ is described in Definition \ref{Def-Hlogk}. 

To state the next theorem in a simple manner, we shall say informally that a rational approximation
$\tilde{a} \in \Q$ 
to a $p$-adic number 
$a \in \Q_p$ is correct modulo
$p^N$ if $\ord_p(\tilde{a} - a) \geq N$, with this notion extended in the obvious way to matrices and
polynomials.

\begin{theorem}\label{Thm-Surprise}
Assume that the residue field of $\W$ is the prime field $\fp$ where $p > 2$, and denote by
$h^{2,0} := \dim H^0 (\XK,\hat{\Omega}_{\XK}^2)$ the geometric genus of $\XK$ (Definition \ref{Def-HodgeNos}).
Let $\tilde{A}$ be an approximation to the matrix for
the Frobenius map $F$ on $\Hlogk \otimes_\W \K$ w.r.t. a basis of $\Hlogk$ which is correct modulo $p^{N+1}$, where $N \geq 2$.
Then the polynomial $\det(1 - p^{-1}\tilde{A} T)$ is correct modulo $p^{N + 1 - h^{2,0} - \lfloor \log_p (k+1)
\rfloor}$.
\end{theorem}

This theorem is surprising since the matrix $p^{-1}\tilde{A}$ will not in general have
entries in $\W$, and so one naively expects a much greater loss of precision when the characteristic
polynomial is computed. For example, when $p > k+1$ one expects naively
that the characteristic polynomial of $p^{-1} \tilde{A}$ is only correct modulo
$p^{N+1 - (m-1)}$ where $m$ is the dimension of $H^2_{rig}(\Uf)$.
In our application to L-functions of elliptic curves, $m$ is much larger than $h^{2,0}$, and so
Theorem \ref{Thm-Surprise} is of great practical use.

Theorem \ref{Thm-Surprise} is an immediate corollary of the next proposition.
 
\begin{proposition}\label{Prop-Surprise}
Let $\B$ be a basis for the $\W$-lattice $\Hlogk$, and hence also the $\K$-vector space $\Hlogk \otimes_\W \K$.
Let $A \in {\rm GL}(m,\K)$ be the matrix for the Frobenius map $F$ acting on $\Hlogk \otimes_\W \K$ w.r.t. this basis.
Let $N \geq 2$ be a positive integer and $\tilde{A} \in {\rm GL}(m,\Q)$ be a matrix such that $\ord_p (\tilde{A} - A) \geq N+1$. Write $\det(1 - A T) = \sum_{\ell = 0}^m a_\ell T^\ell$ and
$\det(1 - \tilde{A} T) = \sum_{\ell = 0}^m \tilde{a}_\ell T^\ell$. Then 
\[ \ord_p(a_\ell - \tilde{a}_\ell) \geq N + i + m(\ell) - \lfloor \log_p (k+1) \rfloor\]
where $m(\ell)$ is defined in the proof of Lemma \ref{Lem-Surprise1}.
\end{proposition}

We shall prove this proposition in several simple steps. 
Using the notation and results in Section \ref{Sec-FHW}, we have an exact sequence
 \[ 0 \rightarrow \Hprim \stackrel{\theta}{\rightarrow} \Hlog \stackrel{\phi}{\rightarrow} H(\Df)(-1) \rightarrow 0 \]
 which defines a mixed Frobenius-Hodge structure (Definition \ref{Def-FHS})
 on $\Hlog$ from  the pure Frobenius-Hodge structures on
 the outer terms. Say that a basis for $\Hprim$ is adapted
to the Hodge filtration if the first $h^{2,0}$ elements form a basis for $H_2$, and the first
$h^{2,0} + h^{1,1}-1$ form a basis for $H_1$. Likewise, say that a basis for $H(\Df)$ is adapted
to the Hodge filtration if the first $g$ elements form a basis for $H_1^\prime$. Say that a basis for $\Hlog$ is adapted
to the Hodge filtration if the first $h_2 - 1$ elements are $\theta$ applied to a basis
of $\Hprim$ which is adapted to the Hodge filtration, and $p^{-1}\phi$ applied to the last $2g$ elements gives a basis
of $H(\Df)$ which is adapted to the Hodge filtration. Note that $m = (h_2  - 1) + 2g$ is the rank of
$\Hlog$.

The matrix $A$ for Frobenius $F$ on a basis of $\Hlog$ which is adapted to the Hodge filtration has the
shape
\[
\left(
\begin{array}{ccccc}
p^2 A_2 & pA_1 &  A_0 & B_1 & B_2\\
0 & 0 & 0 & p^2C_1 & pC_2
\end{array}
\right)
\]
where $A_0,A_2 \in {\rm M}(h_{2} - 1,h^{2,0},\W),\,A_1 \in {\rm M}(h_{2} - 1,h^{1,1}-1,\W),\,C_i \in 
{\rm M}(2g,g,\W),\,B_i \in {\rm M}(h_{2} - 1,g,\W)$ and the $0$s are zero matrices
of the appropriate size.

\begin{lemma}\label{Lem-Surprise1}
Let $\B$ be a basis for $\Hlog$ which is adapted to the Hodge filtration and $A$ the matrix for
Frobenius $F$ on $\Hlog$ w.r.t. the basis $\B$. Let $N \geq 2$ be a positive integer and $\tilde{A} \in {\rm GL}(m,\Q)$ be a matrix such that $\ord_p (\tilde{A} - A) \geq N+1$. Write $\det(1 - A T) = \sum_{\ell = 0}^m a_\ell T^\ell$ and
$\det(1 - \tilde{A} T) = \sum_{\ell = 0}^m \tilde{a}_\ell T^\ell$. Then 
\[ \ord_p(a_\ell - \tilde{a}_\ell) \geq  N + i  + m(\ell)\]
where $m(\ell)$ is defined in the proof.
\end{lemma}
 
\begin{proof} Write
$p^{-1} A = (p^{-1} a_{i,j})$ and
$p^{-1}\tilde{A} = (p^{-1} \tilde{a}_{i,j})$, so $\ord_p(p^{-1}\tilde{a}_{i,j} - p^{-1}a_{i,j}) \geq N \geq 2$. 
From the explicit block-form of the matrix $A$ immediately above, one sees that the matrix $V = (v_{i,j})$ below
is such that
$\ord_p(p^{-1} a_{i,j}),\,\ord_p(p^{-1} \tilde{a}_{i,j}) \geq v_{i,j}$. 
\[
V:= 
\left(
\begin{array}{ccccccccccccccc}
1 & \cdots & 1 & 0 & \cdots & 0 & -1 & \cdots & -1 & -1 & \cdots & -1 & -1 & \cdots & -1 \\
\vdots & & \vdots & \vdots & & \vdots &  \vdots & & \vdots  & \vdots & & \vdots & \vdots & & \vdots\\ 
1 & \cdots & 1 & 0 & \cdots & 0 & -1 & \cdots & -1 & -1 & \cdots & -1 & -1 & \cdots & -1 \\
N & \cdots & N & N & \cdots & N & N & \cdots & N & 1 & \cdots & 1 & 0 & \cdots & 0 \\
\vdots & & \vdots & \vdots & & \vdots &  \vdots & & \vdots  & \vdots & & \vdots & \vdots & & \vdots\\ 
N & \cdots & N & N & \cdots & N & N & \cdots & N & 1 & \cdots & 1 & 0 & \cdots & 0
\end{array}
\right)
\]
Explicitly, the top left-hand $(h_2 - 1) \times h^{2,0}$ sub-matrix of $1$s in $V$ corresponds to lower bounds on the valuation of entries
in the matrix $pA_2$, and so on. Let $1 \leq \ell \leq m$. The $\ell$th coefficient of $\det(1 - Tp^{-1}\tilde{A})$
can be written explicitly as a signed sum of transversal products of the elements $p^{-1} \tilde{a}_{i,j}$ in
$p^{-1}\tilde{A}$, c.f. \cite[Equation (36)]{L}. Each transversal product is determined by a choice of
indices $1 \leq u_1 < \cdots < u_\ell \leq m$ and permutation $\tau \in S_\ell$; namely,
this defines the term ${\rm sign}(\tau) (p^{-1} \tilde{a}_{u_1,u_{\tau(1)}}) \cdots (p^{-1}\tilde{a}_{u_{\ell},u_{\tau(\ell)}})$. According to \cite[Lemma 9.5]{L}, the loss of precision
\[
\ord_p((p^{-1} a_{u_1,u_{\tau(1)}}) \cdots (p^{-1}a_{u_{\ell},u_{\tau(\ell)}}) - (p^{-1} \tilde{a}_{u_1,u_{\tau(1)}}) \cdots (p^{-1}\tilde{a}_{u_{\ell},u_{\tau(\ell)}}))\]
when computing such a product is bounded below by $N + v(u_1,\dots,u_\ell;\tau)$ where
\[
v(u_1,\dots,u_\ell;\tau) := \sum_{i = 1}^{\ell} v_{u_i,u_{\tau(i)}} - \min_{1 \leq i \leq \ell} \{v_{u_i,u_{\tau(i)}}\}.
\]
So for each $\ell$, we must find the minimum value $m(\ell)$ taken by the function $v(u_1,\dots,u_\ell;\tau)$ over different
choices of indices and permutation. By inspection of the matrix $V$, one sees that
\[
m(\ell) = \left\{
\begin{array}{ll}
-\ell + 1 & 1 \leq \ell \leq h^{2,0}\\
-h^{2,0} + 1 & h^{2,0} < \ell \leq h^{2,0} + h^{1,1} - 1 + g\\
-h^{2,0} + 1 + (\ell - (h^{2,0} + h^{1,1} - 1 + g)) & h^{2,0} + h^{1,1} - 1 + g < \ell \leq m.
\end{array}
\right.
\]
The key point is that since $N \geq 2$, one can restrict to transversals of $p^{-1}A$ which factor as
a transversal of the $(h_2 - 1) \times (h_2 - 1)$ sub-matrix $(p A_2,\, A_1,\, p^{-1}A_0)$ times one of the 
$2g \times 2g$ sub-matrix $(pC_1,\,C_2)$, and it is easy to compute the minimum value of
$v(\cdot)$ over such transversals.
To see why we can restrict in this way, suppose that  $1 \leq u_1 < \cdots < u_\ell \leq m$ is a choice of
indices and $\tau \in S_\ell$ permutation such that the corresponding transversal of $p^{-1}A$ contains exactly 
$s \geq 1$ elements from the top-right $h_2 - 1 \times 2g$ sub-matrix of $p^{-1} A$. Then by Lemma
\ref{Lem-Transversals}, the transversal must also contain exactly $s$ elements from the 
bottom-left $2g \times h_2 - 1$ sub-matrix of $p^{-1} A$. 
Hence
$v(u_1,\dots,u_\ell; \tau)$ equals $(N-1)s \geq s$ plus a choice of $\ell - s$ elements lying in
distinct columns (and rows) of the top-left $h_2 - 1 \times h_2 - 1$ sub-matrix of $V$ and the
bottom-right $2g \times 2g$ submatrix of $V$, plus one.  Such a sum cannot be less than the
value $m(\ell)$ given above. 
\end{proof}

\begin{lemma}\label{Lem-Transversals}
Let $M$ be an $m \times m$ matrix over a commutative ring and $M = (M_{ij})_{1 \leq i,j \leq 2}$ a block
decomposition where the block $M_{1,1}$ is square of size $m^\prime$, and the block
$M_{2,2}$ is square of size $m^{\prime \prime} := m - m^\prime$. Then any transversal product of $M$ must contain an equal
number of elements from $M_{1,2}$ and $M_{2,1}$. (See the proof of Lemma \ref{Lem-Surprise1} for
a definition of the term transversal product.)
\end{lemma}

\begin{proof}
We may restrict immediately to the case in which the transversal product is determined by choosing an element
from every row and column of $M$, by extracting a suitable submatrix. For such a transversal product, if
$s$ elements are chosen from $M_{1,2}$ then by considering columns one sees $m^{\prime \prime} - s$ are
chosen from $M_{2,2}$, and by considering rows that $m^\prime - s$ are chosen from $M_{1,1}$. Hence
$m - (s + (m^{\prime} - s ) + (m^{\prime \prime} - s)) = s$ are chosen from $M_{2,1}$.
\end{proof}

\begin{lemma}\label{Lem-Surprise2}
Lemma \ref{Lem-Surprise1} is true for $\B$ any basis of $\Hlog$.
\end{lemma}

\begin{proof}
Let $\B$ be any basis for $\Hlog$ and $A$ the matrix for Frobenius $F$ w.r.t. $\B$. Let
$\B^\prime$ be a basis adapted to the Hodge filtration and $A^\prime$ the matrix for Frobenius
$F$ w.r.t. $\B^\prime$. Of course, the characteristic polynomials of $A$ and $A^\prime$ are the same.
Let $\tilde{A} \in {\rm GL}(\Q,m)$ be such that $\ord_p(\tilde{A} - A) \geq N+1$.
Let $C \in {\rm GL}(m,\W)$ be the change of basis matrix from $\B$ to $\B^\prime$, so
$A^\prime = C^{-1} A C$. Define $D := C^{-1} \tilde{A} C$, so the characteristic polynomials of
$D$ and $\tilde{A}$ are identical. We have $D - A^\prime = C^{-1}(\tilde{A} - A) C$ so
$\ord_p(D - A^\prime) \geq N + 1$. Hence the characteristic polynomials of $D$ and $A^\prime$ agree
as claimed in Lemma \ref{Lem-Surprise1}. Thus so do those of $\tilde{A}$ and $A$.

\end{proof}

\begin{lemma}\label{Lem-Surprise3}
Proposition \ref{Prop-Surprise} is true for bases $\B$ of $\Hlogk$ which satisfy the following property:
There exists a basis
$\B^\prime$ of $\Hlog$ with the change of basis matrix from $\B$ to $\B^\prime$ a diagonal
matrix $E = {\rm diag}(p^{n_1},\dots,p^{n_m})$ with $n_i \leq \lfloor \log_p(k+1) \rfloor$.
\end{lemma}

Note that such bases exist since we know that $\Hlogk/\Hlog$ is a finite abelian
$p$-group killed by multiplication by $p^{ \lfloor \log_p(k+1) \rfloor}$.

\begin{proof}
Define $G := E^{-1} \tilde{A} E$ and $H := E^{-1} A E$. Then $H$ is the matrix for Frobenius $F$ w.r.t. a basis
of the lattice $\Hlog$. Moreover, $G - H = E^{-1} (\tilde{A} - A) E$ and so $\ord_p(G - H) \geq (N+1) - \lfloor \log_p(k+1) \rfloor$. 
The characteristic polynomials of $G$ and $H$ are the same as those of $\tilde{A}$ and
$A$, respectively, that is $\sum_{i = 0}^m \tilde{a}_i T^i$ and
$\sum_{i = 0}^m a_i T^i$. By Lemma \ref{Lem-Surprise2} (with notation ``$A$ and $\tilde{A}$'' replaced by
``$H$ and $G$'', and $N$ replaced by $N - \lfloor \log_p(k+1) \rfloor$) they agree as claimed.
\end{proof}
 
One now deduces Proposition  \ref{Prop-Surprise} from Lemma \ref{Lem-Surprise3} using a similar argument
as that in the proof of Lemma \ref{Lem-Surprise2}.

\section{The inclusions $\Uf \subset \Xf$ and $\Vf \subset \Uf$}\label{Sec-UtoV}

In this section we analyse the kernel and cokernel of the map $H^n_{rig}(\Uf) \rightarrow H^n_{rig}(\Vf)$ which arises in the fibration method.
We restrict to the case in which $n = 2$ and $H^1_{rig}(\Xf) = 0$.
The actual practical computation of this map does not
cause any problems. Since both $\Uf$ and $\Vf$ are affine, cohomology classes can be
represented by $n$-forms defined globally, and the map is just restriction. 

We also
briefly consider the problem of computing the residue map $Res$ in the exact sequence
\[  \cdots \rightarrow  H^n_{cris}(\Xf) \rightarrow  H^n_{cris}((\Xf,\Df)) \stackrel{Res}{\rightarrow} H^{n-1}_{cris}(\Df)(-1)
\rightarrow \cdots.\]
Given a basis for the full rank $\W$-lattice $\Hlog = {\rm Im}(H^n_{cris}(\Xf,\Df) 
\rightarrow H^n_{rig}(U))$, the sublattice $\Hprim \subseteq \Hlog$ can in principle be recovered through an understanding of the residue map (Step 2 in Section \ref{Sec-FMSummary}).

\subsection{Excision exact sequences}

In this section we gather some exact sequences which will be used in Section \ref{Sec-Inj} and have
already been used in some earlier sections.
First, for $Y$ a smooth projective variety over $\fq$ and $Z$ a smooth normal crossings
divisor, we have the ``residue'' sequence
\begin{equation}\label{Seq-Loc1}
 \cdots \rightarrow H^i_{cris}(Y) \rightarrow H^i_{cris}((Y,Z)) \stackrel{Res}{\rightarrow} H^{i-1}_{cris}(Z)(-1) \rightarrow \cdots.
 \end{equation}
 (This follows, in the case in which the pair $(Y,Z)$ lifts to a smooth pair over $\W$ at least, from the analogous
 sequence in algebraic de Rham cohomology
 \cite[Proposition 2.2.8]{AKR} and the comparision theorems \cite[Definition 2.3.3, Proposition 2.4.1]{AKR}.)
Second, for $Y$ a smooth variety over $\fq$ of pure dimension $d$ and $Z \subset Y$ a closed
subscheme we have
the localisation sequence
\[
 \cdots \rightarrow H^i_{Z,rig}(Y) \rightarrow H^i_{rig}(Y) \rightarrow H^i_{rig}(Y \backslash Z) \rightarrow
\cdots .
\]
Here $H^i_{Z,rig}(Y)$ is rigid cohomology of $Y$ with support along $Z$, which is really
just defined to make this sequence exact. When $Z$ is smooth and pure of codimension $e$, one has the Gysin isomorphism \cite{NT}
\[ H^i_{rig,Z}(Y) \cong H^{i - 2e}_{rig}(Z)(-e) \]
which leads to the more useful sequence
\begin{equation}\label{Seq-Loc3}
 \cdots \rightarrow H^{i-2e}_{rig}(Z)(-e) \rightarrow H^i_{rig}(Y) \rightarrow H^i_{rig}(Y \backslash Z) \rightarrow
\cdots .
\end{equation}
Sequences (\ref{Seq-Loc1}) and (\ref{Seq-Loc3}) are related in so much as if one
tensors the first by $\K$ over $\W$ one obtains a special case of the second. We
only need (\ref{Seq-Loc3}) for our application in the next section. Finally, we recall that when
$Y$ is affine $H^i_{rig}(Y) = 0$ for $i  > \dim(Y)$.

\subsection{Injectivity of the map $H^2_{rig}(\Uf) \rightarrow H^2_{rig}(\Vf)$}\label{Sec-Inj}

The propositions in this section makes rigorous the ideas in \cite[Section 9.3.1]{L}. First, we consider
the case of a smooth surface $\Xf$.

\begin{proposition}\label{Prop-Inj}
Let $\Xf$ be a smooth projective surface over $\fq$, with $H^1_{rig}(\Xf) = H^3_{rig}(\Xf) = 0$, 
$\Uf$ an affine piece of $\Xf$ obtained by removing a smooth divisor $\Df$, and $\Vf \subset \Uf$ 
such that $\Zf := \Uf \backslash \Vf$ is pure of dimension $1$ and there is a map $\Vf \rightarrow \Sf \subset \Aff^1$ with smooth fibres. 
Then the map $H^2_{rig}(U) \rightarrow H^2_{rig}(V)$ is an injection.
\end{proposition}

\begin{proof}
Writing out the sequence (\ref{Seq-Loc3}) in full for the pair
$(\Xf,\Df)$ one obtains
\[
\begin{array}{ccccccc}
 0 & \rightarrow & H^0_{rig}(\Xf) & \rightarrow & H^0_{rig}(\Uf) & \rightarrow & 0   \\
    & \rightarrow & 0                     & \rightarrow & H^1_{rig}(\Uf) & \rightarrow & H^0_{rig}(\Df)(-1) \\
    & \stackrel{c}{\rightarrow} & H^2_{rig}(\Xf) & \rightarrow & H^2_{rig}(\Uf) & \rightarrow & H^1_{rig}(\Df)(-1)  \\
    & \rightarrow & 0 & \rightarrow & 0 & \rightarrow & H^2_{rig}(\Df)(-1) \\
    & \rightarrow & H^4_{rig}(\Xf) & \rightarrow & 0 & \rightarrow & 0.
\end{array}
\]
Note that $H^0_{rig}(\Xf) = H^0_{rig}(\Df) =  \K$. The image of the map $c$ is the class
of the divisor $\Df$. Since the cycle class map is injective in codimension one \cite[Page 97]{JT0}, we must have
that $H^1_{rig}(\Uf) = 0$.

Let $\Vf  \subset \Uf$ with $\Zf := \Uf \backslash \Vf$ pure of codimension $1$. We do not assume that $\Zf$ is
smooth, so cannot apply the sequence (\ref{Seq-Loc3}) directly to the pair $(\Uf,\Zf)$. To get around this problem: Let
$P$ denote the singular locus on $\Zf$, a union of points, and $\tilde{\Zf} := \Zf \backslash P$, $\tilde{\Uf} := \Uf \backslash P$. First, $(\Uf,P)$ is a smooth pair, with $P$ of codimension $2$ in $\Uf$. So
(\ref{Seq-Loc3}) for this pair is
\[
\begin{array}{ccccccc}
 0 & \rightarrow & 0 & \rightarrow & H^0_{rig}(\Uf) & \rightarrow & H^0_{rig}(\tilde{\Uf})   \\
    & \rightarrow & 0  & \rightarrow & H^1_{rig}(\Uf) & \rightarrow & H^1_{rig}(\tilde{\Uf}) \\
    & \rightarrow & 0 & \rightarrow & H^2_{rig}(\Uf) & \rightarrow & H^2_{rig}(\tilde{\Uf})  \\
    & \rightarrow & 0 & \rightarrow & 0                       & \rightarrow & H^3_{rig}(\tilde{\Uf}) \\
    & \rightarrow & H^0_{rig}(P)(-2) & \rightarrow & 0 & \rightarrow & 0
\end{array}
\]
Note that $\tilde{\Uf}$ is not affine when $P \ne \emptyset$. Second, $(\tilde{\Uf},\tilde{\Zf})$ is a smooth pair with $\tilde{\Zf}$ of codimension $1$ in $\tilde{\Uf}$ and $\Vf = \tilde{\Uf} \backslash 
\tilde{\Zf}$.
So (\ref{Seq-Loc3}) for this pair is
\[
\begin{array}{ccccccc}
 0 & \rightarrow & 0 & \rightarrow & H^0_{rig}(\tilde{\Uf}) & \rightarrow & H^0_{rig}(\Vf)   \\
    & \rightarrow & 0  & \rightarrow & H^1_{rig}(\tilde{\Uf}) & \rightarrow & H^1_{rig}(\Vf) \\
    & \rightarrow & H^0_{rig}(\tilde{\Zf})(-1) & \rightarrow & H^2_{rig}(\tilde{\Uf}) & \rightarrow & H^2_{rig}(\Vf)  \\
    & \rightarrow & H^1_{rig}(\tilde{\Zf})(-1) & \rightarrow & H^3_{rig}(\tilde{\Uf})  & \rightarrow & 0
\end{array}
\]
Thus the following sequence is exact
\[
 0 \rightarrow H^1_{rig}(\Vf) \rightarrow H^0_{rig}(\tilde{\Zf})(-1)
 \rightarrow H^2_{rig}(\Uf) \rightarrow H^2_{rig}(\Vf) \]
 \begin{equation}\label{Seq-VZ} 
 \rightarrow H^1_{rig}(\tilde{\Zf})(-1)
 \rightarrow H^0_{rig}(P)(-2) \rightarrow 0.
\end{equation}

To calculate $H^i_{rig}(V)$ for $i = 1,2$, we use the assumption that there is a surjective map $V \rightarrow S \subset \Aff^1$ with smooth fibres.
Let $H^i_{rig}(\Vf/\Sf)$ denote relative rigid cohomology for this map, where $i = 0,1,2$. Then
\[ H^1_{rig}(\Vf) \cong E_{2,rig}^{0,1} \oplus E_{2,rig}^{1,0},\, H^2_{rig}(V) \cong E_{2,rig}^{1,1}\]
where
\[ 
\begin{array}{rcl}
E_{2,rig}^{1,i} & := & {\rm Coker}\left( \nabla : H^i_{rig}(\Vf/\Sf) \rightarrow H^i_{rig}(\Vf/\Sf) \otimes \Omega^1_\Sf\right).\\
E_{2,rig}^{0,1} & := & {\rm Ker}\left( \nabla : H^1_{rig}(\Vf/\Sf) \rightarrow H^1_{rig}(\Vf/\Sf) \otimes \Omega^1_\Sf\right).
\end{array}
\] 
As in the proof of \cite[Proposition 8.1]{L} we have $E_{2,rig}^{1,0} \cong H^1_{rig}(\Sf)(-1)$ and
one sees $H^1_{rig}(\Sf) \cong H^0_{rig}(\tilde{\Zf})$. Thus (\ref{Seq-VZ}) begins
\[ 0 \rightarrow E_{2,rig}^{0,1} \oplus H^1_{rig}(\Sf)(-1) \rightarrow H^1_{rig}(\Sf)(-1) \rightarrow H^2_{rig}(\Uf) 
\rightarrow \cdots \]
and
we deduce that $E_{2,rig}^{0,1} = 0$ and $H^2_{rig}(U) \hookrightarrow H^2_{rig}(V)$.

Our conclusion is that the following sequence is exact
\[ 0 \rightarrow H^2_{rig}(\Uf) \rightarrow H^2_{rig}(\Vf)  \rightarrow H^1_{rig}(\tilde{\Zf})(-1)
 \rightarrow H^0_{rig}(P)(-2) \rightarrow 0\]
and in particular the map $H^2_{rig}(U) \hookrightarrow H^2_{rig}(V)$ is an injection.
\end{proof}

Next, we consider the case which arises in our application; that of a quasi-smooth hypersurface
in a weighted projective space.

\begin{proposition}\label{Prop-Inj2}
As in Section \ref{Sec-deJong}, let $\Xf := \Yf/G$ where $G$ is a finite group of order
prime to $p$ and
$\Yf$ a smooth surface in $\pr^3_{\sfq}$. Let $\Ef$ be the divisor at infinity of
$\Yf$, assumed smooth and fixed by $G$, and $\Uf := \Uftil/G$ where 
$\Uftil := \Yf \backslash \Ef$.
Furthermore, let $\Vf \subset \Uf$ be such that $\Zf := \Uf \backslash \Vf$ is pure of dimension $1$, and
there is a map $\Vf \rightarrow \Sf \subset \Aff^1$ with smooth fibres. Then the map
$H^2_{rig}(\Uf) \rightarrow H^2_{rig}(\Vf)$ is an injection.
\end{proposition}

\begin{proof}
The proof is as before, once one has established that $H^1_{rig}(\Uf) = 0$. Now
$H^1_{rig}(\Yf) = 0$ since $\Yf$ is a smooth hypersurface in $\pr^3_{\sfq}$. The sequence
(\ref{Seq-Loc3}) for the pair $(\Yf,\Ef)$ reveals that $H^1_{rig}(\Uftil) = 0$. Hence
$H^1_{rig}(\Uf) = H^1_{rig}(\Uftil)^G = 0$, as required.
\end{proof}

\subsection{Computing the residue map}\label{Sec-Residue}

In principle the residue map can be computed by using explicit \v{C}ech-de Rham complexes, following
the approach of Gerkmann \cite{RG}. 
In unpublished notes, the author refined Gerkmann's work to obtain an explicit formula
for the residue map in the case in which $\XW := \pr_\W^2$ and $\DW$ is a curve. Extending
this to the case $\XW := \pr_\W^3$ and $\DW$ is a surface proved too challenging for the present author, although plausible explicit expressions were obtained after much work. Note that the author carried out these latter calculations in an attempt to recover the lattice ${\rm Im}(H^2_{cris}(\Df) \rightarrow
H^2_{rig}(\Df))$ from the construction in \cite{AKR} of the lattice ${\rm Im}(H^3_{cris}((\Xf,\Df))
\rightarrow H^3_{rig}(\Xf \backslash \Df))$ via the residue map.

The author's experience with these calculations suggest that the explicit construction of the residue
map in the case $\XW$ is a surface and $\DW$ is a smooth divisor is in principle possible, but perhaps
not desirable in practice unless $H^2_{rig}(\Uf)$ has much larger dimension than
 ${\rm Im}(H^2_{rig}(\Xf) \rightarrow H^2_{rig}(\Uf))$. In our calculations
in Section \ref{Sec-Ranks} the former had dimension only two greater than the latter. 

\section{The Fuchsian basis problem}\label{Sec-Fuchs}

This section is entirely expositional.
We introduce the Fuchsian basis problem, and explain its relevance to our algorithm.

\subsection{Hilbert's 21$^{st}$ problem and Fuchsian bases}\label{Sec-FBP}

Let $F$ be a field of
characteristic zero and consider a linear differential system
\begin{equation}\label{Eqn-DiffSys}
 \frac{d C(y)}{dy} + B(y) C(y) = 0,
 \end{equation}
where $B(y)$ is an $m \times m$ matrix with entries in the rational function field $F(y)$, and the solution matrix
$C(y)$ is an invertible matrix over some differential extension field of $F(y)$. If
$H(y)$ is any invertible matrix of rational functions, we may transform (\ref{Eqn-DiffSys}) into
an equivalent differential system in which $B$ is replaced by
\begin{equation}\label{Eqn-BH}
 B_{[H]} := H^{-1} B H + H^{-1} \frac{d H}{dy}.
 \end{equation}
For an irreducible polynomial $p(y) \in F[y]$ let $F(y)_p := F(y)[1/p]$ and $F(y)_{(p)}$ denote
the local ring at $p$, i.e., the extension of $F[y]$ in which all 
irreducible polynomials {\it excluding} $p(y)$ have been inverted.
We shall say that (\ref{Eqn-DiffSys}) is {\it regular} at $p(y) \in F[y]$ if
there exists a matrix $H \in GL(m,F(y)_p)$ such that $B_{[H]}$ has entries in $F(y)_{(p)}$.
We shall say that (\ref{Eqn-DiffSys}) is {\it regular singular} at $p(y) \in F[y]$ if it is not
regular at $p(y)$ but 
there exists a matrix $H \in GL(m,F(y)_p)$ such that $p(y)B_{[H]}$ has entries in $F(y)_{(p)}$.
We say that the system is regular/regular singular at infinity, if is regular/regular singular
at $y$, after replacing the matrix $B(y)$ by $-1/y^2 B(1/y)$. We call (\ref{Eqn-DiffSys})
{\it regular singular} if it is either regular or regular singular at all irreducible polynomials $p(y)$ and
infinity. 

For simplicity of exposition, let us henceforth assume that the
system (\ref{Eqn-DiffSys}) is regular singular at infinity, and let $P(y)$ be the product of all irreducible
polynomials
$p(y)$ at which it is regular singular. Let us say $P(y)$ defines the finite singular locus.
We shall say that (\ref{Eqn-DiffSys}) is {\it Fuchsian} if there exists
a matrix $H \in GL(m,F[y,1/P(y)])$ such that $P(y)B_{[H]}$ has entries in $F[y]$ of degree less than $\deg(P(y))$. We call $B_{[H]}$ of this form
a {\it Fuchsian matrix} for the differential system (\ref{Eqn-DiffSys}). 
Fuchsian systems are of course regular singular. 
 Establishing whether the
converse holds when $F$ is the field of complex numbers is closely related to 
Hilbert's 21$^{st}$ problem,
though not the statement of it. 
We require a solution to the problem:
\vspace{0.2cm}

\hspace{-\parindent}{\bf Fuchsian Basis Problem.} {\it Give an algorithm which takes as input a regular singular
differential system (\ref{Eqn-DiffSys}) which is regular singular at infinity and has
finite singular locus defined by the polynomial $P(y)$, and either
\begin{itemize}
\item{Gives as output a change of basis matrix 
$H(y) \in GL(m,F[y,1/P(y)])$ such that $B_{[H]}$ defined in (\ref{Eqn-BH}) is a Fuchsian
matrix for (\ref{Eqn-DiffSys}),}
\item{Terminates and declares correctly that the system is not Fuchsian.}
\end{itemize}}

Hilbert's 21st problem has a chequered but fascinating history.
When $F$ is the field of complex numbers, Fuchsian bases are known to exist for regular singular differential systems
under various conditions, such as irreducibility.
For $m = 3$ there are examples of reducible systems
which are regular singular but not Fuchsian; Bolibruch has made a very explicit study of this case
\cite{AAB}. For $m = 2$, all regular singular
systems over algebraically closed fields are Fuchsian, and an effective algorithm exists for finding Fuchsian bases (Dekkers's algorithm \cite{WD}). 

\subsection{Picard-Fuchs systems}\label{Sec-PF}

Let $\VK \rightarrow \SK \subset \pr_\K^1$ be a smooth family of $\K$-varieties of relative dimension $n-1$. Associated to this family is a free module  $H^{n-1}_{dR}(\VK/\SK)$
of finite rank $m$, say,  over the coordinate ring $A := \Gamma(\SK,\Oh_{\SK})$, 
and an additive and Leibniz linear map (Gauss-Manin connection)
\[ \nabla: H^{n-1}_{dR}(\VK/\SK) \rightarrow H^{n-1}_{dR}(\VK/\SK) \otimes_A \Omega^1_{A}. \]
Choose a coordinate function $y$ on $\SK$, and basis $e_1,\dots,e_m$ for
$H^{n-1}_{dR}(\VK/\SK)$. Then with respect to the choice of bases $\{ e_i \}$ and $\{ e_i \otimes dy\}$
for the domain and codomain of $\nabla$, the map acts as
\[ \frac{d}{dy} + B(y) : A^m \rightarrow A^m\]
for some matrix $B(y)$ with entries in $A$. The differential system associated to the
Gauss-Manin connection and choice of initial basis $\{e_i\}$ and coordinate $y$ is defined to be (\ref{Eqn-DiffSys}). This system is regular singular, and has rational local exponents, by the regularity and
local monodromy theorems \cite{NK1}. The method developed in \cite[Section 4]{L} for computing in
the cokernel of $\nabla$, and controlling the loss of $p$-adic precision which accrues, assumes
that a basis has been chosen such that the matrix  $B(y)$ for the Gauss-Manin connection is Fuchsian.

\section{Hyperelliptic curves over the rational function field}\label{Sec-Hyper}

In this section we tie up various ends which were left loose in \cite{L}. Specifically, 
in that paper the author made a conjecture, the truth of which would improve the practical
performance of the fibration method for hyperelliptic curves. This conjecture is proved in
Section \ref{Sec-Conj74}. Second, the algorithm assumed that certain local
monodromy assumptions held true \cite[Section 7.4]{L}. In practice, it is easy to check computationally whether these assumptions hold, and they always do. We prove that these assumptions
are true for the elliptic curves we consider in Section \ref{Sec-Ranks} by an explicit calculation (Section \ref{Sec-ConnMat}). Finally, in Section \ref{Sec-Complexities} we state some results on the complexity of the fibration method, incorporating the various refinements in this paper.

\subsection{Notation}\label{Sec-HyperNotation}

Let us first introduce the notation which we shall use in this section. Let
$q$ be a power of an odd prime $p$, and $\bar{Q}(x,y) \in \fq[x,y]$ have degree
$2g + 1$ in the variable $x$ for some integer $g \geq 1$. Denote by
$\bar{\Delta}(y)$ the Sylvester resultant of the polynomials $\bar{Q}$ and $\partial \bar{Q}/
\partial x$ with respect to the variable $x$. We assume $\bar{\Delta}(y)$ is non-constant. Let
$Q(x,y) \in \W[x,y]$ be some lift of the polynomial $\bar{Q}(x,y)$, and
$\Delta$ the Sylvester resultant of $Q$ and $\partial Q/\partial x$ w.r.t. $x$.
We need to choose $Q(x,y)$ so that $\deg(\Delta) = \deg(\bar{\Delta}) > 0$. This can certainly
be done under some hypotheses, which are stated in the results below.
Let 
\[ \Uf := \mbox{{\rm Spec}}\left( \fq[x,y,z]/(z^2 - \bar{Q}(x,y)\right),\]
and let $\Vf \rightarrow \Sf$ be the family obtained
by taking ``{\rm Spec}'' of the homomorphism
\[ \fq[y,1/\bar{\Delta}(y)] \hookrightarrow \fq[x,y,1/\bar{\Delta}(y),z]/(z^2 - \bar{Q}(x,y)).\]
We assume that $\Uf$ is smooth.

In our discussion of complexity estimates in Section \ref{Sec-Complexities}, we assume that $\Uf$ compactifies to a quasi-smooth surface $\Xf$ in a weighted projective space subject to certain further conditions, exactly as in Section \ref{Sec-Surfaces}, so that we may use $\W$-lattices in $H^2_{rig}(\Uf)$ to speed-up computations. However, this compactification plays no role in Sections \ref{Sec-Conj74} and \ref{Sec-ConnMat}.

\subsection{A conjecture on the relative Frobenius matrix}\label{Sec-Conj74}

In this section we prove \cite[Conjecture 7.4]{L}. This gives a useful practical
improvement to the performance of the fibration method for hyperelliptic curves over
function fields $\fq(y)$.

\subsubsection{Statement of conjecture 7.4}

We recall some notions from \cite[Section 7]{L}, alerting the reader again to the fact that our
notation has changed.
Define $A^\dagger$ to be the weak completion of the coordinate ring of the lifting
$\SK$ of the base $\Sf$. That is
\[ A^\dagger := \left\{ \sum_{i = -\infty}^\infty \frac{a_i(y)}{\Delta(y)^i}\,|\,
\deg(a_i) < \deg(\Delta),\, \ord_p(a_i) - \varepsilon |i| \rightarrow \infty \mbox{ for some }\varepsilon
 > 0 \right\}.\]
There is a lifting of the $p$th power map from $\fq[y,1/\bar{\Delta}(y)]$ to $A^\dagger$ which
sends $x \mapsto x^p$ and $y \mapsto y^p$, unique up to a choice of sign. Let
us denote this map by $\sigma_{A^\dagger}$. There is a rank $2g$ free module $H^1_{rig}(\Vf/\Sf)$ over $A^\dagger$ with basis given by the classes of the forms $x^i dx/z$ for $i = 0,1,\dots,2g-1$. There is a
map $F : H^1_{rig}(\Vf/\Sf) \rightarrow H^1_{rig}(\Vf/\Sf)$ which is $\sigma_{A^\dagger}$-linear, called
the $p$th power (relative) Frobenius map. We represent this map $F$ by a matrix $F(y)$ w.r.t. our choice of basis.
\vspace{0.2cm}

\hspace{-\parindent}{\bf \cite[Conjecture 7.4]{L}} {\it The matrix $F(y)$ has a pole of finite order at
$y = \infty$.}
\vspace{0.2cm}

\subsubsection{Elliptic curves}

Assume that $\bar{Q}(x,y) := x^3 + \bar{a}(y)x + \bar{b}(y)$ for some $\bar{a}(y),\bar{b}(y) \in \fq[y]$. Then $\bar{\Delta} = 4\bar{a}^3 + 27\bar{b}^2$. We first
consider a special case.

\begin{theorem}\label{Thm-Conj74special}
Assume that $\deg(\bar{a}) \leq 4e$ and $\deg(\bar{b}) \leq 6e$ for some $e \geq 1$, and $\deg(\bar{\Delta}) = 12e$.
Then the entries in matrix $F(y)$ when expanded locally as Laurent series in $\K((1/y))$ have degree
in $y$ at most $(p+1)e$; that is, the pole order of $F(y)$ at $y = \infty$ is bounded by
$(p+1)e$.
\end{theorem}

\begin{proof}
First, lift the elliptic curve equation to get $z^2 = x^3 + a(y) x + b(y)$ where $a,b \in \W[y] \subset \K(y)$ with $\deg(a) \leq 4e,\, \deg(b) \leq 6e$ and $\Delta := 4a^3 + 27b^2$ such that $\deg(\Delta) = 12 e$.
Note that any lifting satisfying the degree restrictions on $a$ and $b$ will be suitable.
We make the change of variables $v := 1/y, u := v^{2e} x$ and $w := v^{3e}z$ to get
an equation $w^2 = u^3 + \tilde{a}(v) u + \tilde{b}(v)$. Here $\tilde{a}(v) := v^{4e} a(1/v) \in \W[v]$ and
$\tilde{b}(v) := v^{6e} b(1/v) \in \W[v]$. The condition on the degree of the discriminant ensures us that
$\tilde{\Delta}(v) := 4 \tilde{a}^3 + 27 \tilde{b}^2 = v^{12e} \Delta(1/v)$ does not vanish
at $v = 0$, even modulo $p$. So the lifted curve over $\W[v]$ has good reduction at $v = 0$, even modulo $p$.
Recall that $d$ is the universal derivation
of the function field of the generic fibre of the lifted curve over the base field $\K(y) = \K(v)$. We see that 
\begin{equation}\label{Eqn-DiffChange}
\frac{dx}{z} = v^e \frac{du}{w},\, \frac{xdx}{z} = v^{-e} \frac{udu}{w} \mbox{ and }
\frac{du}{w} = y^e \frac{dx}{z},\, \frac{udu}{w} = y^{-e} \frac{x dx}{z}.
\end{equation}
The matrix $F(y)$ expresses the action of the $p$th power Frobenius map on the basis
$dx/z$ and $xdx/z$. Let $F_{ij}(y)$ be the $(i,j)$th entries in $F(y)$, for $1 \leq i,j \leq 2$. 
These are $p$-adic functions which converge on some open annulus of outer radius $1$ around $y = \infty$. We wish to show that
they only have a finite polynomial part in $y$; that is, they converge on the punctured open unit disk
around $y = \infty$. Letting $\sigma$ denote the Frobenius action on functions and differentials we have
\[ \sigma\left( \frac{dx}{z} \right) = \sigma\left( v^e \frac{du}{w} \right) = v^{ep} \sigma \left( \frac{du}{w} \right).\]
The expression $\sigma \left( \frac{du}{w} \right)$ may be computed using Kedlaya's method 
\cite[Section 7.3.3]{L}, and
written modulo exact differentials as a linear combination 
\[ \sigma \left( \frac{du}{w} \right) \equiv A_{1,1} (v) \frac{du}{w} + A_{1,2}(v) \frac{u du}{w}.\]
The key point is that $A_{1,1}(v), A_{1,2}(v) \in \K[[v]]$; thus these functions do not have poles
at $v = 0$. The reason for this is that since the lifted curve modulo $p$ is smooth at $v = 0$, Kedlaya's reduction
formulae do not introduce poles at this point. (Kedlaya's explicit method only introduces poles at singular fibres; indeed
one can adapt this proof to show poles do not occur at singular fibres with potentially good reduction.) Similarly, we see that
\[  \sigma \left( \frac{x dx}{z} \right) = v^{-ep} \sigma \left( \frac{u du}{w} \right)
\equiv v^{-ep} \left( A_{2,1} (v) \frac{du}{w} + A_{2,2}(v) \frac{u du}{w} \right) \]
 where $A_{2,1}(v),A_{2,2}(v) \in \K[[v]]$. Hence using (\ref{Eqn-DiffChange}) to get back to
our original basis we find
\[ 
\begin{array}{ll}
F_{1,1} = y^{-ep} A_{1,1}(1/y) y^e, & F_{1,2} = y^{-ep} A_{1,2}(1/y) y^{-e},\\
F_{2,1} = y^{ep} A_{1,2}(1/y) y^e, & F_{2,2} = y^{ep} A_{2,2}(1/y) y^{-e}
\end{array}
\]
where $A_{i,j} \in \K[[1/y]]$. Thus the matrix $F(y)$ has no positive powers of $y$ in the
$(1,1)$ and $(1,2)$ positions, degree in $y$ bounded by $(p+1)e$ in the $(2,1)$ position, and
degree bounded by $(p-1)e$ in the $(2,2)$ position.
\end{proof}


The bounds in the theorem above on the pole order for the elliptic curves considered in Section
\ref{Sec-Ranks} are $(p+1)e$ where $p = 7$ and $e \in \{1,2,3,4,5\}$, that is $8,16,24,32,40$, and the
actual pole orders observed were $6,12,18,24, 30$, respectively.

We now consider the general case.

\begin{theorem}\label{Thm-Conj74general}
Assume $\deg(\bar{a}) \leq d/3$ and $\deg(\bar{b}) \leq d/2$ where the
discriminant $\bar{\Delta} := 4\bar{a}^3 + 27\bar{b}^2$ has degree $d \geq 1$.
Then the entries in matrix $F(y)$ when expanded locally as Laurent series in $\K((1/y))$ have degree
in $y$ at most $\lfloor (p+1) d/12 \rfloor$; that is, the pole order of $F(y)$ at $y = \infty$ is bounded by
$\lfloor (p+1) d/12 \rfloor$.
\end{theorem}

\begin{proof}
Lift the curve as before, preserving the degree of the discriminant.
Let $m$ be the smallest integer such that $12|md$ where $d = \deg(\Delta(y)))$. Define
$\tilde{y} := y^{1/m}$ and consider the generic fibre of the original lifted curve as being defined by the equation $z^2 = x^3 + a(\tilde{y}^m)x + b(\tilde{y}^m)$ over the totally ramified extension $K(\tilde{y})/K(y)$. Then
$\deg_{\tilde{y}}(\Delta(\tilde{y}^m)) = 12e$ for some integer $e \geq 1$. Now proceed
as in the proof of Theorem \ref{Thm-Conj74special}, via the change of variables $v := 1/\tilde{y}, u := v^{2e} x$ and $w := v^{3e}z$,
to deduce that the Frobenius matrix $F(\tilde{y})$ has degree in $\tilde{y}$ bounded by
$(p+1)e$. But this matrix converges on an open annulus around infinity, and so must
have entries in $K((1/y))$. Thus we deduce that the pole order at $y = \infty$ of
$F(y)$ is bounded by $\lfloor (p+1) e/m \rfloor = \lfloor (p+1)  d/12 \rfloor$.
\end{proof}

Applying the above theorem to \cite[Examples 9.1, 9.2]{L}, one obtains bounds on the pole order
of $\lfloor (17 + 1) 26/12 \rfloor = 39$ and $\lfloor (5+1) 62/12 \rfloor = 31$, respectively, which are exactly the pole
orders observed experimentally.

\subsubsection{Hyperelliptic curves}

The same technique can be used to prove \cite[Conjecture 7.4]{L} in greater generality; that is,
first consider the case when
the fibre $y = \infty$ is smooth, using an explicit change of basis, and then make a totally ramified extension of $K(y)$ to reduce to this case. We need the following lemma.

\begin{lemma}\label{Lem-HomDis}
Let $g \geq 0$ and $\Z[a_2,a_4,\dots,a_{2(2g+1)}]$ be a graded polynomial ring, where
$\wt(a_i) := i$. Then the discriminant of the polynomial $x^{2g+1} + a_2 x^{2g} + 
a_4 x^{2g-1} + \cdots + a_{4g} x + a_{2(2g+1)}$ is homogeneous of degree
$2(2g+1)2g$.
\end{lemma}

\begin{proof}
For $\lambda \in \Q^*$ the zero locus of the discriminant is invariant under the
map $a_{2i} \mapsto \lambda^i a_{2i}\,(1 \leq i \leq 2g+1)$, since on the original polynomial this transformation composed
with $x \rightarrow \lambda x$ is just multiplication by $\lambda^{2g+1}$. So the discriminant
is weighted homogeneous w.r.t. the given weights. To compute the degree, specialise all but the constant
coefficient to zero and use the explicit expression as the discriminant as the determinant of a
Sylvester matrix.
\end{proof}

The next theorem subsumes Theorem \ref{Thm-Conj74general}.

\begin{theorem}\label{Thm-Conj74MoreGeneral}
Let $\bar{Q}(x,y) \in \fq[x,y]$ be monic of degree $2g+1$ in the variable $x$ for some integer $g \geq 1$.
Let $\bar{\Delta}(y)$ be the discriminant of the hyperelliptic curve $z^2 = \bar{Q}(x,y)$ over
$\fq(y)$, and $d  := \deg(\bar{\Delta}) > 0$. Assume that when one writes $\bar{Q}(x,y) = 
x^{2g+1} + \bar{a}_2(y) x^{2g} + \bar{a}_4(y) x^{2g-1} + \cdots + \bar{a}_{2(2g+1)}(y)$ we have 
\[ \deg_y(\bar{a}_i(y)) \leq d \frac{i}{2(2g + 1)2g}.\]
Then the entries in the matrix $F(y)$ when expanded locally as Laurent series in $\K((1/y))$ have degree
in $y$ at most $\lfloor d(2g-1)(p+1)/(4(2g+1)g) \rfloor$; that is, the pole order of $F(y)$ at $y = \infty$ is bounded by
$\lfloor d (2g-1)(p+1)/(4(2g+1)g) \rfloor$.
\end{theorem}

\begin{proof}
Lift the polynomial $\bar{Q}(x,y)$ to a polynomial $Q(x,y) \in \W[x,y]$ preserving the discriminant
degree. Lemma \ref{Lem-HomDis} and the assumptions on the degrees of the coefficients
$\bar{a}_i(y)$ assures us that any lifting which preserves the degree bounds on the coefficients
will be suitable.
Let us first assume that $d = 2(2g+1)2ge$ for some integer $e$.
Define $v:= 1/y$, $u := v^{2e} x$ and
$w := v^{(2g+1)e} z$. Let $\tilde{Q}(u,v) \in K[u,v]$ be defined as
$\tilde{Q}(u,v) := v^{d/2g} Q(v^{-2e} u, 1/v)$. Then the equation $w^2  = \tilde{Q}(u,v)$ defines
the generic fibre of the original lifted curve over $K(v)$ w.r.t. the new choice of variables. The discriminant
$\tilde{\Delta}(v)$ of the polynomial $\tilde{Q}(u,v)$ equals
$v^d \Delta(1/v)$. Since we preserved the degree of the discriminant when lifting, $\tilde{\Delta}(0) \ne 0$, even modulo $p$, and the fibre at $v = 0$ is smooth, even modulo $p$.
We can compute the action of $F$ on the new basis $ u^i du/w$ for $0 \leq i < 2g$ using
Kedlaya's algorithm. Let the $(i,j)$th coefficients of the matrix thus obtained
be $A_{i,j}(v)$. Since the fibre at $v = 0$ is smooth, we have that the local expansions of the functions
$A_{i,j}(v)$ at $v = 0$ lie in $K[[v]]$. 

Let $F_{i,j}(y)$ be the $(i,j)$th coefficient of the matrix for the action of $F$ on the original basis
$x^i dx /z$, $0 \leq i < 2g$. Using the formulae
\[ \frac{x^i dx}{z} = v^{(2g+1)e - 2e(i+1)} \frac{u^i du}{w},\, \frac{u^j du}{w} = y^{(2g+1)e - 2e(j+1)} \frac{x^j dx}{z}\]
we see that
\[ F_{i,j}(y) = y^{-p[(2g+1)e - 2e(i+1)]} y^{(2g+1)e - 2e(j+1)} A_{i,j}(1/y) \]
\[ = y^{2e(p(i+1)-(j+1)) - (p-1)e(2g+1)} A_{i,j}(1/y).\]
Thus the maximum degree in $y$ which occurs is bounded by
\[ 2e(2gp - 1) - (p-1)e(2g+1) = e(2g-1)(p+1).\]
The general case is treated by first making the change of variable
$\tilde{y} := y^{1/m}$ where $m$ is the smallest integer such that
$2(2g+1)2ge = md$ for some integer $e$.
One deduces that the degree in $\tilde{y}$ of the matrix
$F(\tilde{y})$ is bounded by $e(2g-1)(p+1)$. Hence the degree
of $F(y)$ in $y$ is bounded by $\lfloor e(2g-1)(p+1)/m \rfloor = \lfloor d (2g-1)(p+1)/(4(2g+1)g) \rfloor$. 
\end{proof}

Applying the above theorem to \cite[Example 9.3]{L} one obtains a bound of 
$\lfloor 28 (4 - 1)(11 + 1)/(4(4 + 1)2) \rfloor = 25$ and the experimentally observed
pole order is $21$.

\subsection{The Gauss-Manin connection for elliptic curves}\label{Sec-ConnMat}

Write $\bar{Q}(x,y) = x^3 + \bar{a}(y) x + \bar{b}(y)$ and
$Q(x,y) = x^3 + a(y) x + b(y)$, with other notation as in Section \ref{Sec-HyperNotation}. Denote
by $B(y)$ the matrix for the Gauss-Manin connection on $H^1_{dR}(\VK/\SK)$ with
respect to the coordinate function $y$ and initial basis the cohomology classes of
$dx/z$ and $x dx/z$, see Section \ref{Sec-PF} or \cite[Section 7.3.1]{L}.

\begin{proposition}\label{Prop-ExpEllConn}
The matrix for the Gauss-Manin connection with respect to our choice of
basis, acting on the left on column vectors, is
\[
B(y) = 
\frac{1}{\Delta(y)}
\left(
\begin{array}{ll}
-a^2 \frac{d a}{d y} - \frac{9}{2} b \frac{d b}{d y}
&  -a^2 \frac{d b}{d y} + \frac{3}{2} ab\frac{d a}{d y} \\
 -3a\frac{d b}{d y} + \frac{9}{2} b \frac{d a}{d y}
& 
a^2 \frac{d a}{d y} + \frac{9}{2} b \frac{d b}{d y}
\end{array}
\right).
\]
\end{proposition}

\begin{proof}
An explicit calculation direct from the definition which uses the Sylvester matrix of $Q(x,y) =
x^3 + a(y)x + b(y)$ and $\frac{\partial Q}{\partial x}$ with respect to $x$, and
Kedlaya's formula for pole reduction of differential forms \cite{KK0}.
\end{proof}

Let $\bar{\K}$ be an algebraic closure of $\K$, and $R := \{\gamma \in \bar{\K}\,|\,\Delta(\gamma) = 0\}$.
Let ${\rm M}(K[y],2)$ be the ring of $2 \times 2$ matrices over $K[y]$.
Define $\beta(y) \in {\rm M}(K[y],2)$ by $\beta(y) := \Delta(y)B(y)$.

\begin{lemma}
The matrix $\beta(\gamma)$ is nilpotent for all $\gamma \in R$.
\end{lemma}

\begin{proof}
First, from Proposition \ref{Prop-ExpEllConn} observe $\Tr(\beta(y)) = 0$. 
An explicit calculation reveals that 
\[ \det(\beta(y)) = - \frac{1}{4}\Delta(y)\left( a \left(\frac{d a}{d y}\right)^2 + 3 
\left(\frac{d b}{d y}\right)^2 \right).\]
So at a root $\gamma$ of $\Delta(y) = 0$ we have $\det(\beta(\gamma)) = 0$, and hence
$\det(T - \beta(\gamma)) = T^2$, as required.
\end{proof}

\begin{proposition}\label{Prop-NilRes}
Assume that $\Delta(y)$ is squarefree, and let $\gamma \in R$. 
Let $t_\gamma := y - \gamma$, a uniformizing
parameter at $\gamma$, $K(\gamma)((t_\gamma))$ be the field of Laurent series, and embed $K[y,1/\Delta(y)] \hookrightarrow K(\gamma)((t_\gamma))$. Then 
\[ B(t_\gamma) = r_\gamma t_\gamma^{-1} + s_\gamma(t_\gamma),\, s_\gamma(t_\gamma) \in 
{\rm M}(K(\gamma)[[t_\gamma]],2)\]
where the residue matrix $r_\gamma \in {\rm M}(K(\gamma),2)$ is nilpotent.
\end{proposition}

\begin{proof}
We have
\[ B(y) = \frac{\beta(y)}{\Delta(y)} = \frac{\beta(y)}{\Delta^\prime(y)}
\sum_{\gamma \in R} \frac{1}{y - \gamma}.\]
Let $\gamma \in R$. Then $B(t_\gamma)$ can be written as claimed where
$r_\gamma :=  \beta(\gamma)/\Delta^\prime(\gamma)$.
Since $\beta(\gamma)$ is nilpotent, it follows that $r_\gamma$ is nilpotent.
\end{proof}

Thus under the hypothesis that $\Delta(y)$ is squarefree, the Gauss-Manin has regular singular points at the roots of $\Delta(y) = 0$, with nilpotent residue matrix. So we have proved in an entirely explicit manner the following well-known result.

\begin{corollary}
When the discriminant $\Delta(y)$ is non-constant and squarefree, the family of elliptic curves defined by the equation $z^2 = x^3 + a(y)x + b(y)$ has regular singularities and unipotent local monodromy at the roots of $\Delta(y) = 0$.
\end{corollary}

We now consider the local monodromy at infinity. First, let us consider the case which arises
in our calculations in Section \ref{Sec-Ranks}.

\begin{lemma}
Assume that $\deg(a) \leq \deg(b)/2$ and $\deg(\Delta) = 2 \deg(b)$.
Then we have $\deg_y (B(y)) \leq -1$, and the coefficient
of $y$ in the local expansion at zero of $B(y)$ has the form
\[
\frac{1}{27}
\left(
\begin{array}{cc}
-\frac{9}{2}\deg(b) & \star \\
0 & \frac{9}{2} \deg(b)
\end{array}
\right)
\]
where $\star = 0$ when $\deg(a) < \deg(b)/2$. When $\deg(a) = \deg(b)/2$ we have
\[ \star = \frac{\alpha^2}{\beta} \left( -\deg(b) + \frac{3}{2} \deg(a) \right) = -\frac{\alpha^2}{4\beta}
\deg(b),\] 
where $\alpha$ and $\beta$ are the leading coefficients of $a$ and $b$, respectively.
\end{lemma}

\begin{proof}
This can be read off from the connection matrix in Proposition \ref{Prop-ExpEllConn}.
\end{proof}

When $\deg(b) \equiv 0 \bmod{6}$, as in Section \ref{Sec-Ranks}, the residue matrix
has integer eigenvalues $\pm \deg(b)/6$. Note that since the residue matrix is not
prepared in this case, one cannot apply directly the precision loss bounds deduced
in \cite[Section 4]{L} when reducing the ``polynomial part'' of differential forms. 
However, by virtue of our proof of \cite[Conjecture 7.4]{L}, very little
reduction needs to be done on the ``polynomial part'', and naive estimates are sufficient.

When $\deg(a) \geq \deg(b)$ and $\deg(\Delta) = 3\deg(a)$, the matrix 
$B(y)$ is such that $\deg_y (B(y)) \leq -1$, and one can explicitly read off the residue
matrix. However, in the case $\deg(b)/2 < \deg(a) < \deg(b)$ the
entry $B_{1,2}(y)$ does not have a simple pole at infinity and one needs to make
a change of basis. We explain how to do this in the proof of Proposition \ref{Prop-Change}, although
we do not examine the local monodromy at infinity. Note that we have not implemented this
change of basis routine
in our algorithm as it requires an extension of the base field. The author does not know whether
a Fuchsian basis exists for the general case without first extending the base field.

Assume $\deg(b)/2 <  \deg(a) < \deg(b)$ and $d := \deg(\Delta) = \max\{2\deg(b),3 \deg(a) \} > 0$.
When $\deg(a) \leq 2\deg(b)/3$ define $m := 2\deg(a) - \deg(b)$, and
when $\deg(a) > 2\deg(b)/3$ define $m:= \deg(b) - \deg(a)$. Note that $m+1$ is the pole
order at infinity of $B_{1,2}(y)$. Assume that $\Delta(y)$ is squarefree, and also that
either $B_{1,1}(y)$ or $B_{2,1}(y)$ is non-zero.

\begin{proposition}\label{Prop-Change}
Let $L \supseteq \K$ be such that $\Delta(y)$ has a root $\gamma \in L$ and a
factor of degree $m$ defined over $L$. Then there exists a basis for 
$H^1_{dR}(\VK/\SK) \otimes_{\K} L$ such that the matrix for the Gauss-Manin connection
$\nabla$ is Fuchsian (see Section \ref{Sec-FBP} for the definition of this term).
\end{proposition}

\begin{proof}
After a change of variable $y \mapsto y - \gamma$ we may assume $\gamma = 0$.
We make a further change of variable $t:=1/y$. The matrix for the connection
$\nabla$ with respect to the same basis but different choice of coordinate function
$t$ is $\tilde{B}(t) := -1/t^2 B(1/t)$. Define $\tilde{\Delta} := t^d \Delta(1/t)$. Then from
Proposition \ref{Prop-ExpEllConn} one sees
\[ \tilde{B}(t) = 
\left(
\begin{array}{ll}
\frac{c(t)}{t \tilde{\Delta}(t)} & \frac{d(t)}{t^{m+1} \tilde{\Delta}(t)}\\
\frac{t^{m-1}e(t)}{\tilde{\Delta}(t)} &  -\frac{c(t)}{t \tilde{\Delta}(t)}
\end{array}
\right),\mbox{ where }c(t),d(t),e(t) \in L[t] \mbox{ , $e(t)$ or $c(t)$ non-zero}.
\]
Moreover, since $B(y)$ had a simple pole at $y = 0$, one sees that
$\deg(c) \leq \deg(\tilde{\Delta})$, $\deg(d) 
\leq \deg(\tilde{\Delta}) + m$ and
$\deg(e) \leq \deg(\tilde{\Delta}) - m$. 
Now make a change of basis using the matrix
\[
H(t) := 
\left(
\begin{array}{ll}
g(t) & \alpha(t) \\
0 & t^m
\end{array}
\right)
\]
where $g(t)$ is a factor of $\tilde{\Delta}(t)$ of degree $m$, and
$\alpha(t) \in L[t]$ will shortly be chosen. One computes that
$\tilde{B}_{[H]}$ (see (\ref{Eqn-BH})) equals
\[
\left(
\begin{array}{ll}
\frac{c - \alpha e}{t\tilde{\Delta}} + \frac{g^\prime}{g}
& \frac{2c \alpha - e \alpha^2 + d}{t\tilde{\Delta}g} + \frac{\alpha^\prime}{g} - \frac{m \alpha}{t g}\\
\frac{eg}{t \tilde{\Delta}} & \frac{e \alpha - c}{t \tilde{\Delta}} + \frac{m}{t}
\end{array}
\right).
\]
Choose $\alpha \in K[t]$ with $\deg(\alpha) < \deg(g) = m$ such that
\[ -e \alpha^2 + 2c \alpha + d = g f \]
for some $f \in \W[t]$; that is, solve the quadratic equation $-e X^2 + 2c X + d  = 
0 \bmod{g}$. This is always possible without extending the field further: the
discriminant of the equation is $4c^2 + 4ed$, which is zero modulo $\tilde{\Delta}(t)$ since
the residue matrices of $\tilde{B}(t)$ are nilpotent at the roots of $\tilde{\Delta}(t)$ (by Proposition \ref{Prop-NilRes}).
With this choice of $\alpha$ one sees that the entries in the matrix
$\tilde{B}_{[H]}$ are of the form $\star/t \tilde{\Delta}(t)$ where $\deg(\star) \leq \deg(\tilde{\Delta})$.
\end{proof}

\subsection{Complexity estimates}\label{Sec-Complexities}

All of the complexity estimates in this section are for deterministic algorithms.
Let $\bar{Q}(x,y) \in \fq[x,y]$ with $\deg_x(\bar{Q}) = 2g + 1$ and $h := \deg_y(\bar{Q})$.
From \cite[Theorem 1]{LW} we see that one may compute the L-function of
a hyperelliptic curve $z^2 = \bar{Q}(x,y)$ over $\fq(y)$ in 
$(p g h \log(q))^{C}$ bit operations for some absolute constant $C \geq 1$. For $g$ fixed and $h/p$ bounded, 
\cite[Theorem 8.6]{L} tells us we may compute the L-function in
$\Oh(p^{1 + \varepsilon} \log(q)^{3 + \varepsilon} h^{4 + \varepsilon})$ bit operations, for
any $\varepsilon > 0$. Let us now assume $\bar{Q}(x,y)$ satisfies the hypothesis in the statement of
Theorem \ref{Thm-Conj74MoreGeneral}.
For $h$ fixed,  using ideas similar to those in the last paragraph
of \cite[Section 8.3]{L} and Theorem \ref{Thm-Conj74MoreGeneral} one can show that the L-function
may be computed in $\Oh(p^{1 + \varepsilon} \log(q)^{3 + \varepsilon}g^{7 + \varepsilon})$ bit operations, for any $\varepsilon > 0$. The main point is that the bound in Theorem \ref{Thm-Conj74MoreGeneral} does not grow with $g$. This last bound does not use any of the refinements in
this paper other than the proof of \cite[Conjecture 7.4]{L}.

Assume now that $g$ is fixed.
The difficulty lies in proving bounds when
$h/p$ is allowed to grow. The problem is that using the method in \cite{L}, even combined with our
proof of \cite[Conjecture 7.4]{L}, one can only show that the matrix of the $p$th power Frobenius on $H^2_{rig}(\Vf)$ has valuation bounded
below by some absolute constant times $-h/p$. This lower bounds impacts on the complexity
when one computes the characteristic polynomial, because of the possibility of precision loss. By computing
Frobenius on the basis elements of the lattice $\Hlogk$ in Section \ref{Sec-Surfaces} one gets a matrix whose valuation is bounded below by some
absolute constant times $-\log_p(h)$, and so there is no problem with precision loss when one
computes the characteristic polynomial. Unfortunately, as yet the author has not been able to prove useful bounds during another step of the refined algorithm; namely after reduction in the cokernel of
$\nabla$, when an element of $\Hlogk \otimes_\W \K$ written
w.r.t a given generating set for $H^2_{rig}(\Vf)$
is written as a linear combination of the chosen basis for $\Hlogk$. In practice, this does not cause any problems: this step can be accomplished by multiplication by an explicit matrix with entries in an algebraic number field, and one can measure the loss of precision simply by computing the valuation of the matrix at the prime $p$. In the examples computed in Section \ref{Sec-Ranks}, this loss of precision was
zero; note that this precision loss can be computed right at the start of the algorithm, before
various precisions are chosen. Assuming the loss of precision at this step is bounded by some function of $\log_p(h)$, then
one gets a complexity of $\Oh(p^{1 + \varepsilon} \log(q)^{3 + \varepsilon}h^{4 +
\varepsilon})$, for any $\varepsilon > 0$. One must assume here that all of the conditions needed for the
various steps are met. For elliptic curves with discriminant of degree $12e$ it is sufficient that $p > 3$ and 
$p$ does not divide $e$, 
$Q(x,y) = x^3 + \bar{a}(y)x + \bar{b}(y)$ with
$\deg(a) \leq \deg(b)/2$, the discriminant $4\bar{a}^3 + 27\bar{b}^2$ is squarefree, the affine surface
$z^2 = \bar{Q}(x,y)$ is smooth, the projective surface defined by the affine equation $z^{6e} = \bar{Q}(x^{2e},y)$ is smooth, and torsion is not encountered during the construction of the lattice
$\Hlogk$ in Section \ref{Sec-Surfaces}. We stress that the output of the algorithm
is always provably correct.



\section{Ranks of elliptic curves}\label{Sec-Ranks}

For each $d \in \{6,12,18,24,30\}$, the author selected in a verifiably random manner
$1000$ elliptic curves of the form
\[ z^2 = x^3 + a(y) x + b(y) \]
where $a,b \in \f_7[y]$ with $\deg(b) = d$ and $\deg(a) \leq d/2$, subject to the 
condition that the discriminant is squarefree and does not vanish at $y = 0$. We imposed the further
condition that the projective surface defined by the affine equation $z^d =  x^d + a(y) x^{d/3} + b(y)$ is smooth, see Note \ref{Note-ExtraCond}.
 The L-functions
of the curves were computed using the fibration method. These are polynomials of
degree $2d-4$ with all reciprocal roots of complex absolute value $7$. The curves and their
L-functions, along with the details of how the curves were generated, will appear
as an attachment in the final version of  this paper. The table below details
the number of curves of each analytic rank.

\[
\begin{array}{|| l | c  c  c  c c||} \hline
 & d = 6 & d = 12 & d = 18 & d = 24 & d = 30\\ \hline
\mbox{analytic rank }0  & 336 & 378 & 449 & 473 & 505 \\ \hline
\mbox{analytic rank }1 & 481 & 489 & 489 & 496 & 482 \\ \hline
\mbox{analytic rank }2 & 162 & 120 &  55  & 31 & 13 \\ \hline
\mbox{analytic rank }3 & 19 &  11  &  7    & 0 & 0 \\ \hline
\mbox{analytic rank }4 &  2 &  2    &   0   & 0 & 0 \\ \hline 
\mbox{analytic rank }>4 & 0 & 0 & 0 & 0 & 0 \\
\hline
\end{array}
\]

Assuming the Tate conjecture, one observes that the percentage with infinitely many rational points is $66.4\%$, $62.2 \%$ (these are
elliptic K3 surfaces, so the Tate conjecture is known here \cite[Theorem 5.6. (b)]{JT}), $55.1\%$, $52.7\%$ and $49.5\%$, respectively --- neatly flattening from the previous experimentally
observed fraction of around two thirds to around the predicted value of one half.

\begin{note}\label{Note-ExtraCond}
The additional generic condition that the projective surface defined by the affine equation $z^d = x^d + a(y) x^{d/3} + b(y)$ is smooth is an artifact of
the method of construction of the $\W$-lattice $\Hlogk$ in Section \ref{Sec-deJong}, and undoubtedly irrelevant to the algorithm and the correctness
of its output. The table for the first $1000$ curves chosen without imposing this condition is as follows. 

\[
\begin{array}{|| l | c  c  c  c c||} \hline
 & d = 6 & d = 12 & d = 18 & d = 24 & d = 30\\ \hline
\mbox{analytic rank }0  & 346 & 379 & 453 & 479 & 497\\ \hline
\mbox{analytic rank }1 & 471 & 494 & 493 & 493 & 493 \\ \hline
\mbox{analytic rank }2 & 162 & 114 &  47  & 28 & 10 \\ \hline
\mbox{analytic rank }3 & 20 &  11  &  7    & 0 & 0 \\ \hline
\mbox{analytic rank }4 & 1 &  2    &   0   & 0 & 0 \\ \hline 
\mbox{analytic rank }>4 & 0 & 0 & 0 & 0 & 0 \\
\hline
\end{array}
\]

Assuming the Tate conjecture, one observes that the percentage with infinitely many rational points here is $65.4\%$, $62.1 \%$, $54.7\%$, $52.1\%$ and $50.3\%$, respectively.
\end{note}

\begin{note}
The restriction that $\deg(a) \leq d/2$, rather than the more natural $\deg(a) \leq 2d/3$,
forces the $j$-invariants of the curves chosen to vanish to a high degree at infinity. So
we have not chosen curves randomly in a dense open subset of all elliptic curves with discriminant
of degree $2d$. However, the curves are chosen uniformly at random from a dense open subset
 of a ``Weiestrass family'' which for $d \geq 12$ has geometric monodromy group
the full orthogonal group
 $O(2d - 4)$ \cite[Theorem 10.2.13]{NK3}. What we are really testing experimentally are
the predictions of Katz for such families as one allows the degree of the conductor to grow
\cite[Introduction]{NK2}. These
predictions include that asymptotically exactly one half will have infinitely many rational points. 
\end{note}

\end{document}